\documentclass[12pt]{article}
\usepackage{amssymb,amsmath,amsthm}

\newtheorem{theorem}{Theorem}

\newtheorem{lemma}{Lemma}
\newtheorem{claim}{Claim}
\newtheorem{problem}{Problem}
\newtheorem{prop}{Proposition}
\def \nin {\noindent}

\def \bsk {\bigskip}
\def \hsp {\hspace{0.8em}}
\def \Fm {\mbox{${\mathcal Fm}$}}
\def \fequiv {\!\!\equiv}  
\def \Mm {\mathfrak M}
\def \slashm {\!\slash }

\title{Transposition of variables is hard to axiomatize}

\author{Hajnal Andr\'eka, Istv\'an N\'emeti \ and \
  Zsolt Tuza\footnote
	{Also affiliated with the Department of Computer Science
	and Systems Technology, University of Pannonia,
	Veszpr\'em, Hungary.}\\
  \\
\normalsize
HUN-REN Alfr\'ed R\'enyi Institute of Mathematics\\
\normalsize
 H--1053 Budapest, Re\'altanoda u.~13--15, Hungary }

\usepackage{tikz}
\usetikzlibrary{shapes.geometric, arrows.meta,patterns,decorations.pathreplacing,angles,quotes}

\pgfdeclarepatternformonly{mynewdots}{\pgfqpoint{-1pt}{-1pt}}{\pgfqpoint{1pt}{1pt}}{\pgfqpoint{1pt}{1pt}}
{%
	\pgfpathcircle{\pgfqpoint{0pt}{0pt}}{.3pt}%
	\pgfusepath{fill}%
}%

\begin{document}

\maketitle

\begin{abstract} 
The function $p_{xy}$ that interchanges two logical variables $x,y$ in formulas is hard to describe in the following sense. Let $F$ denote the Lindenbaum--Tarski formula-algebra of a finite-variable first-order logic, endowed with $p_{xy}$ as a unary function. We prove that each equational axiom system for the equational theory of $F$ has to contain, for each finite $n$, an equation that contains together with $p_{xy}$ at least $n$ algebraic variables, and each of the operations $\exists, =, \lor$. This gives an answer to a problem raised by Johnson [{\em J. Symb.\ Logic\/}] in 1969: the class $RPEA_{\alpha}$ of representable polyadic equality algebras of a finite dimension $\alpha\ge 3$ cannot be axiomatized by adding finitely many equations to the equational theory of representable cylindric algebras of dimension $\alpha$. Consequences for proof systems of finite-variable logic and for defining equations of polyadic equality algebras are given. 

The proof uses a family of nonrepresentable polyadic equality algebras ${\cal A}_n$ that are more and more nearly representable as $n$ increases: their $n$-generated subalgebras as well as their proper reducts are representable. The lattice of subvarieties of $RPEA_{\alpha}$ is investigated and new open problems are asked about the interplay between the transposition operations and about generalizability of the results to infinite dimensions.
%
%
%
%
%
\end {abstract}


\section{Introduction} 
\label{s:1}

Manipulating individual variables is a common practice in predicate logic. 
For example, if we change the variables in a sentence to any others, the truth value does not change in so far that we change distinct variables to distinct ones.
This fact is used, for example, in the proof of the prenex normal form theorem of first-order predicate logic. %

We deal with finite-variable first-order logic in this paper. Throughout, except in Section \ref{inf-s}, we assume that  $\alpha\ge 3$ is a finite ordinal.
Finite-variable logic has an extensive literature. It provides insights about the nature of ordinary first-order logic, and, e.g., it is widely used in computer science.%
\footnote{See, e.g., \cite{Da99, Do01, Hod96, RySh22}.}

By {\it transposition of variables} $v,w$ in first-order logic we mean the function that to each formula (possibly with free variables) associates the formula obtained by exchanging all occurrences of the variables $v,w$ in it.
For example, transposing variables $v,w$ in $\exists v(v=w\land R(w,u))$ yields $\exists w(w=v\land R(v,u))$. 

Transposing variables $v,w$ defines a function on the set of all formulas.  It  will be convenient to consider a natural algebra, the {\it algebra of formulas} $\Fm$, on the set of formulas. 

To define this, let us assume that we have infinitely many $\alpha$-place relation symbols in the vocabulary of non-logical symbols, and possibly other relation symbols of smaller arity, but we do not have function or constant symbols. 
Let $V=\{ v_i : i<\alpha\}$ be the set of variables, and let the logical connectives be $\lor, \neg, \exists v_i, v_i=v_j$ for $i,j<\alpha$.
For traditional reasons, we denote the natural operations of $\Fm$ defined by these connectives as $+, -, c_i, d_{ij}$, respectively.  For example, if $\varphi$ is a formula, then $c_i(\varphi)$ is $\exists v_i\varphi$ in $\Fm$. 
Again for traditional reasons, let $p_{ij}$ denote the function on $\Fm$ that transposes the variables $v_i$ and $v_j$ in formulas, and let $\Fm^+$ denote the algebra $\Fm$ endowed with the functions $p_{ij}$ for $i,j<\alpha$ as unary operations. We call the operations $c_i$ and $d_{ij}$ as cylindrifications and diagonal constants, respectively, for their geometrical meaning.%
\footnote{We consider the formula $v_i=v_j$ as a nullary logical connective. For details about the formula-algebra, see \cite{HeMoT:II} or \cite{AGyNS}.}

The following equations (P1)--(P7) are true in $\Fm^+$.  Below, we use $x,y$ as algebraic variables, i.e., they range over all elements of $\Fm$,  
and $i,j,k,l<\alpha$. 
Further, $\tau=[i,j]$ where $[i,j]$ denotes the function that interchanges $i,j$ and leaves all other ordinals fixed.
\begin{description}  
	\item[(P1)]\label{p1} $p_{ij}(x+y) = p_{ij}x + p_{ij}y $.
	\item[(P2)]\label{p2} $p_{ij}(-x) = -p_{ij}x$.
	\item[(P3)]\label{p3} $p_{ij}(c_kx) = c_{\tau(k)}p_{ij}x$.
	\item[(P4)]\label{p4} $p_{ij}d_{kl} = d_{\tau(k)\tau(l)} $.
	\item[(P5)]\label{p6} $p_{ij}p_{kl}x = p_{\tau(k)\tau(l)}p_{ij}x$.
	\item[(P6)]\label{p7} $p_{ij}p_{ij}x = x$.
	\item[(P7)]\label{p5} $p_{ii}x = x$.
\end{description}
Equations (P1)--(P4) come from the definition of interchanging $v_i$ and $v_j$ in a formula as going through the formula from left to right, when we encounter the symbol $v_i$ we change it to $v_j$ and proceed, when we encounter $v_j$ we change it to $v_i$ and proceed, and in all other cases we just proceed.
Equation (P7) expresses that interchanging $v_i$ and $v_i$ this way amounts to doing nothing. Finally, equations (P5)--(P6) come from Bjarni J\'onsson's defining relations for the group of finite permutations of a set in \cite{J62}: all the true equations in terms of transpositions $[i,j]$ in that group can be derived from equations corresponding to (P5)--(P7). 
In fact (P1)--(P7)  are all the equations true in $\Fm^+$ in the sense that an equation is true in $\Fm^+$ if and only if it follows from (P1)--(P7), see Theorem \ref{poz-thm} in Section \ref{log-s}.

However, we are interested in those equations that are true semantically, not just syntactically. For example, the formulas $\exists v_0R(v_0)$ and $\exists v_1R(v_1)$ are semantically equivalent, but not syntactically. Let $\Fm\slash\fequiv$ denote the tautological, or semantic, formula-algebra. That is, $\varphi\equiv\psi$ iff%
\footnote{``iff" abbreviates ``if and only if"} the formula  $\varphi\leftrightarrow\psi$ is valid. It is not hard to see that $\equiv$ is a congruence with respect to the transposition functions $p_{ij}$, i.e., $\varphi\equiv\psi$ implies that $p_{ij}(\varphi)\equiv p_{ij}(\psi)$.%
\footnote{See, e.g., \cite{AmSa} or \cite{AGyNS}. For the notion of tautological formula-algebras, see, e.g., \cite[Sec.\ 4.3]{HeMoT:II} or \cite{AGyNS}.}
The real question is what equations are true in $\Fm^+\slashm\fequiv$.
For example, the following equation is not true in $\Fm^+$, but it is true in $\Fm^+\slashm\fequiv$.
\begin{description}  
	\item[(P8)]\label{p8} $p_{ij}(x\cdot d_{ij}) = x\cdot d_{ij}$.
\end{description}

James Johnson proved in 1969 that the equations true in $\Fm^+\slashm\fequiv$ are harder to describe than those valid in $\Fm^+$:

\bsk
\noindent {\bf Theorem A} (Johnson \cite{Jo}).\label{Jo-thm}
{\it The set of equations valid in $\Fm^+\slashm\fequiv$ is not finitely axiomatizable.}

\bsk
However, this complexity might not be due to the transposition operations, because of the following theorem of James Donald Monk.

\bsk
\noindent {\bf Theorem B} (Monk \cite{Mo}).\label{Mo-thm}
{\it The set of equations valid in $\Fm\slash\fequiv$ is not finitely axiomatizable.}

\bsk

The previous two theorems raise the question whether the transposition operations are finitely axiomatizable {\emph{over}} the set of equations valid in \mbox{$\Fm\slash\fequiv$}. 

\bsk

\noindent {\bf Problem C} {\rm (\cite[second part of Problem 2]{Jo}, \cite[Problem 5.8]{HeMoT:II})}\label{main-p}
	Is there a finite set $\Sigma$ of equations such that the equations true in $\Fm^+\slashm\fequiv$ are exactly those that are derivable from $\Sigma$ together with all the equations true in
	  \mbox{$\Fm\slash\fequiv$}\,? 
\bsk

Problem C is equivalent to asking whether the equational theory of \mbox{$\Fm^+\slashm\fequiv$} can be axiomatized by a set of equations in which the transposition operations occur only finitely many times. In fact, the conjecture was that there is such a set, namely (P1)--(P8) can be taken for $\Sigma$ in Problem C. We prove in this paper that the transposition operations are much harder to describe over the semantic formula-algebra than this: 

\begin{theorem}\label{th:0} Each equational axiom system for the equations true in \hbox{$\Fm^+\slashm\fequiv$} must contain, for each natural number $n$, an equation in which $n$ distinct algebraic variables occur together with at least one transposition operation, at least one cylindrification operation and at least one diagonal constant.
\end{theorem}

Theorem \ref{th:0} provides an answer to Problem C. 
Neither (P1)--(P8), nor any finite $\Sigma$, nor any infinite set in which finitely many algebraic variables occur, nor any infinite set containing infinitely many distinct algebraic variables in which the transposition operations occur infinitely many times but not together with both cylindrifications and diagonal constants, etc., can axiomatize the equational theory of $\Fm^+\slashm\fequiv$. 

Problem C was first asked in 1969 as \cite[second part of Problem 2]{Jo} and it was repeated in 1985 as \cite[Problem 5.8]{HeMoT:II}. In \cite{Jo}, partial results are proved in the direction that (P1)--(P8) might be taken as a finite axiom set for the transposition operations, and the question about (P1)--(P8) is explicitly asked in \cite[p.\ 348]{Jo} and, for $\alpha=3$, as \cite[Problem 5.7]{HeMoT:II}. 
Andr\'eka and N\'emeti, in the unpublished manuscripts \cite{AN84} and \cite{APEA87}, showed that (P1)--(P8) cannot be taken for $\Sigma$ in Problem C, and Andr\'eka and Tuza announced in \cite{ATu88} a negative answer to Problem C.  The present paper contains the first full published proof, and also for a theorem stronger than the one announced earlier.
Discussion of the problem and referencing these results can be found in \cite[p.\ 236 and Remark 5.4.40]{HeMoT:II}, in \cite[p.\ 725]{amn} as well as in \cite[p.\ 204]{hh}.

Interest in this problem might stem from a desire for understanding the role of individual variables in first-order logic.  
Two different algebraizations of the semantics of first-order logic are Alfred Tarski's cylindric algebras and Paul Halmos' polyadic algebras; in the former manipulating variables are sort of derived operations while in the latter manipulating variables are explicitly treated. Halmos \cite[p.\ 28]{Ha} writes that ``The exact relations between polyadic algebras and cylindric algebras are of considerable technical interest; they are still in the process of being clarified".%
\footnote{This relationship is still being investigated, see, e.g., \cite{fer, FerAU18, FerND22, Say13, Say13a}.}
In fact, formally, Problem C 
 was asked about the relationship between polyadic and cylindric algebras. We now define these algebras and relate them to $\Fm^+\slashm\fequiv$ and $\Fm\slash\fequiv$.

Polyadic equality set algebras of dimension $\alpha$ are generalizations of Boolean set algebras. 
Their universes consist of subsets of $\alpha$-dimensional spaces instead of arbitrary (1-dimensional) sets, and they have, besides the Boolean set operations of union and complementation, extra operations that come from the geometric nature of an $\alpha$-dimensional space. 

A {\it polyadic equality set algebra of dimension $\alpha$} and with {\it base set} $U$ is an algebra
$$ {\cal A} = \langle A, +, -, C^U_i, D^U_{ij}, P_{ij} \rangle _{i,j < \alpha} $$
where $\langle A, +, - \rangle$ is a Boolean set algebra with unit the $\alpha ^{\rm th}$ power $U^{\alpha}$ of $U$,
hence the elements of the {\em universe}\/ $A$ are subsets of $U \times \cdots \times U$ and $+,-$ are the set-theoretic union and complementation with respect to ${U}^\alpha$, the {\em 
cylindrifications}\/ $C^U_i$ are unary operations acting as
$$ C^U_i(X) = \{(u_l)_{l < \alpha} \in {U}^\alpha \mid (\forall j\ne i)u_j = u_j ^{\prime},\mbox{ for some }(u_l ^{\prime})_{l < \alpha} \in X 
 \} $$
for every $X \in A$, the $D^U_{ij}$ are the {\em diagonal constants}
$$D_{ij}^{U}=\{(u_l)_{l < \alpha} \in {U}^\alpha \mid u_i = 
u_j \}$$
and the {\em polyadic transposition operations}\/ $P_{ij}$ are unary operations 
with
$$\begin{array}{rcl}
	P_{ij}(X) = \{(u_l)_{l < \alpha} \in {U}^\alpha & \mid 
	&  u_i = u_j ^{\prime} \; \wedge \; u_j = u_i ^{\prime} \wedge \; (\forall l < \alpha , \: i \neq l \neq j ) u_l = u_l ^{\prime}, \\
	 &  & \mbox{for some }(u_l ^{\prime})_{l < \alpha} \in X    
	\} .
\end{array}$$
The set $A$ is supposed to be closed under the operations $+$, $-$, 
$C^U_i$ and $P_{ij}$, and it has to contain the constants $D^U_{ij}$.
Of the three extra-Boolean operations, the cylindrifications and the diagonal constants depend on the base set $U$, while the transposition operations do not depend on the base set. We often omit the superscript $U$ referring to the base set when this is not likely to cause confusion.

The geometric meaning of $C_i^U$ is translation parallel with the $i$th axis, the diagonal constant $D^U_{ij}$ is the $ij$-diagonal set, and $P_{ij}$ is (orthogonal) reflection to this $ij$-diagonal.  See Figure \ref{pse2-fig}. 

\begin{figure}[h]
	\begin{center}

		\begin{tikzpicture}[scale=1]
			\draw[->] (0,0) -- (4,0) node[right] {$1$};
			\draw[->] (0,0) -- (0,4) node[above] {$0$};
			\draw[-] (0,0) -- (4,4) node[right] {$D_{01}^U$};
			\draw[-] (0,1.4) -- (4,1.4) {};
			\draw[-] (0,2.2) -- (4,2.2) {};
			\draw[pattern=mynewdots] (0,1.4) rectangle (4,2.2);
			\draw[white] (4,1.4) -- (4,2.2) {};
			\node[right] at (4,1.8) {$C_{1}^{U}X$};
			\draw[thick] (1,1.4) -- (2.8,2.2) {};
			\draw[thick] (1.4,1) -- (2.2,2.8) {};
			
			\draw[->]  (3.5,2.6) to [bend left=25]  (2.5,2);
			\node[right] at (3.4,2.7) {$X$};
			
			\draw[->]  (1.5,3.5) to [bend right=75]  (2,2.6);
			\node[right] at (1.5,3.5) {$P_{01}X$};
			\draw[decorate,decoration={brace,amplitude=5pt}] (-0.3,0) -- (-0.3,3.95) node [black,midway,xshift=-0.4cm] {$U$};
			
			\draw[decorate,decoration={brace,amplitude=5pt,mirror}] (0, -0.3) -- (3.95, -0.3) node [black,midway,yshift=-0.4cm] {$U$};

		\end{tikzpicture}
		\caption{The extra-Boolean operations of a $Pse_2$}
	\end{center}
	\label{pse2-fig}
\end{figure}
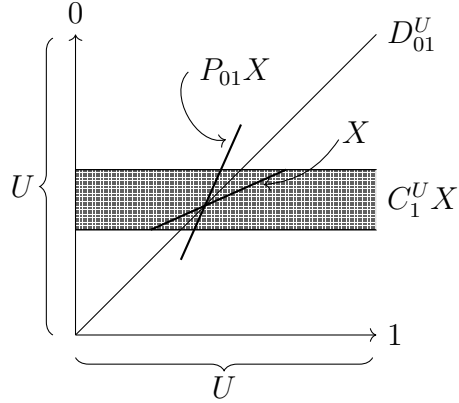

In Codd's relational database model, $P_{ij}$ corresponds to interchanging the $i$th and $j$th columns, this can be expressed by Codd's renaming, inserting and deleting operations. For more on this, see \cite{Due}.

A {\it cylindric set algebra of dimension $\alpha$} and with {\it base set} $U$ is an algebra
$$ {\cal A} = \langle A, +, -, C^U_i, D^U_{ij}\rangle _{i,j < \alpha} $$
where $\langle A, +, - \rangle$ is a Boolean set algebra with unit $U^{\alpha}$.
The set $A$ is supposed to be closed under the operations $+$, $-$, and
$C^U_i$, and it has to contain the constants $D^U_{ij}$.

The classes of $\alpha$-dimensional polyadic equality set algebras and  cylindric set algebras are denoted by $Pse_{\alpha}$ and $Cs_{\alpha}$, respectively. Polyadic and cylindric  algebras are intensively investigated%
\footnote{See, e.g., \cite{amn, Geo79, HeMoT:II, hh, KhSa, Sag, Say11}.}.

In its formal language, the operations of a polyadic equality set algebra are denoted as those of $\Fm^+$, i.e., with $+,-,c_i,d_{ij},p_{ij}$. Similarly, the operations of a cyindric set algebra are denoted with $+,-,c_i,d_{ij}$.
It is known that an equation is true in $\Fm^+\slashm\fequiv$ if and only if it is true in $Pse_{\alpha}$, and an equation is true in $\Fm\slash\fequiv$ if and only if it is true in $Cs_{\alpha}$, see Proposition \ref{pse-prop} in Section \ref{log-s}. 
  
Johnson's and Monk's previously mentioned classic results were formulated about polyadic and cylindric algebras of finite dimension $\alpha \ge 3$. Monk \cite{Mo} proves that the equational theory of $Cs_{\alpha}$ is not finitely axiomatizable, and Johnson \cite{Jo} proves that the equational theory of $Pse_{\alpha}$ is not finitely axiomatizable. 
By an {\it equational axiom set} for $Pse_{\alpha}$ we mean a set of equations true in $Pse_{\alpha}$ from which all equations true in $Pse_{\alpha}$ are derivable.
The second part of Problem 2 in \cite{Jo}, let us call it Problem 2b, raises the problem whether or not $Pse_{\alpha}$ is finitely axiomatizable over $Cs_{\alpha}$; this is  equivalent to asking whether or not $Pse_{\alpha}$ has an equational axiom set in which the transposition
operations occur only finitely many times.
In these terms, our Theorem \ref{th:0} gives the following negative answer to \cite[Problem 2b]{Jo}. 
Let $\omega$ denote the smallest infinite ordinal, so $\alpha<\omega$ means that $\alpha$ is finite.

\begin{theorem}   \label{t:1}
	In every equational axiom set for  $Pse_{\alpha}$ $(3 \leq \alpha < \omega)$ there is a set\/ $\{e_n\}_{n<\omega}$ of axioms such that, for each\/ $n$,  more than\/ $n$ distinct algebraic variables occur in\/ $e_n$ together with some  transposition $p_{ij}$, some diagonal constant $d_{kl}$ and some cylindrification $c_m$.
\end{theorem}

A variant of this result states that the single polyadic operation $P_{01}$ is 
equally hard to describe, i.e., the class of algebraic structures of the 
form
$$ {\cal A} = \langle A, +, -, C_i, D_{ij}, P_{01} \rangle _{i,j < \alpha} $$
has no equational axiom set in which $p_{01}$ occurs finitely many times. 

We note that the first part of \cite[Problem 2]{Jo} is already answered in the negative in \cite{AAPAL97}. In light of \cite{AGyN25}, the first part of \cite[Problem 2]{Jo} is  equivalent to asking whether $Pse_{\alpha}$ has an axiom set in which the diagonal constants occur only finitely many times. Thus Theorem \ref{t:1} above gives an answer to both parts of \cite[Problem 2]{Jo}.

Theorem \ref{th:0} follows from Theorem \ref{t:1} 
by Proposition \ref{pse-prop} in Section \ref{s:7}. 
We deduce Theorem \ref{t:1} from a slightly different result given below.
An algebra  ${\cal A^\prime}$ 
is {\it $n$-generated} if there is a subset $X$ of its universe $A'$ such that $|X|\le n$ and there is no proper subset of $A^\prime$ that contains $X$ and is closed under all operations of ${\cal A}^\prime$. 
We call an algebra {\it polyadic-type} if it is similar to members of $Pse_{\alpha}$, i.e., if it has operations $+,-,c_i,d_{ij}, p_{ij}$ for $i,j<\alpha$, with the corresponding arities.
The {\it cylindrification-free}, the {\it diagonal-free}, and the {\it transposition-free, or  cylindric, reducts} of an algebra $\langle A,+,-,c_i,d_{ij}, p_{ij}\rangle_{i,j<\alpha}$ are $\langle A,+,-,d_{ij}, p_{ij}\rangle_{i,j<\alpha}$, $\langle A,+,-,c_i, p_{ij}\rangle_{i,j<\alpha}$, and $\langle A,+,-,c_i,d_{ij}\rangle_{i,j<\alpha}$, respectively. 

\goodbreak

\begin{theorem}   \label{t:2}
For every\/ $\alpha$ and\/ $n$ $(3 \leq \alpha < \omega$, $n < \omega)$ 
there exists a polyadic-type algebra ${\cal A}$
with the following \vspace{0.06in} properties.

\medskip

\noindent
{\bf (i)}\quad ${\cal A}$ is not isomorphic to a member of $Pse_{\alpha}$.

\medskip

\noindent
{\bf (ii)}\quad Each $n$-generated subalgebra of ${\cal A}$ is 
isomorphic to a member of $Pse_{\alpha}$.

\medskip

\noindent
{\bf (iii)}\quad 
The cylindrification-free, the diagonal-free, and the transposition-free reducts of ${\cal A}$ are all 
isomorphic to subalgebras of the corresponding reducts of members of $Pse_{\alpha}$.
\end{theorem}

We call a polyadic-type algebra {\it representable} iff all equations true in $Pse_{\alpha}$ are true in it. It is known that representable algebras are exactly the subdirect products of polyadic set algebras, and a representable algebra is isomorphic to a member of $Pse_{\alpha}$ iff the formula $x\ne 0\to c_0c_1\dots c_{\alpha-1}x=1$ is true in it (where $0$ and $1$ denote the Boolean $0$ and $1$, respectively).%
\footnote{This is true because the class of algebras isomorphic to a member of $Pse_{\alpha}$ is a discriminator class. For more details we refer to \cite[Sec.\ 2.6.4]{hh} and \cite[Sec.\ 2.7]{AGyNS}.}

To see that Theorem \ref{t:2} indeed implies Theorem \ref{t:1}, 
let $x_1,\dots,x_n$ be algebraic variables and let $\Sigma_n$ denote all the equations true in $Pse_{\alpha}$  that contain at most the variables $x_1,\dots, x_n$ or in which not all kinds of extra-Boolean operations occur. Then $\Sigma_n$ is valid in ${\cal A}$ by (ii) and (iii). On the other hand, 
the formula $x\ne 0\to c_0c_1\dots c_{\alpha-1}x=1$ is true in ${\cal A}$ by (ii). Hence, ${\cal A}$ is not representable, by (i). Since $\Sigma_n$ is true in ${\cal A}$, this means that $\Sigma_n$ is not an equational axiom set for $Pse_{\alpha}$. This latter then means that in each equational axiom set for $Pse_{\alpha}$ there is an equation that contains more than $n$ variables and contains at least one of each extra-Boolean operations.

Let ${\cal A}=\langle A,f_k\rangle_{k\in K}$ be an arbitrary algebra. A reduct of ${\cal A}$ is $\langle A,f_k\rangle_{k\in J}$  where $J\subseteq K$. In this paper, by a {\it set algebra} we mean a subalgebra of a reduct of a polyadic equality set algebra. The {\it base set} of a set algebra is the base set of the corrresponding polyadic equality set algebra.

Sections \ref{s:2}--\ref{s:6} of the paper are devoted to the proof of Theorem \ref{t:2} and its variant where 
just one polyadic operation $P_{01}$ is considered. 
(The former does not imply the latter directly.) 
In Sections \ref{s:2} and \ref{s:3}, we describe the construction of the algebras. 
For each odd prime power $p\ge 3$ we construct a polyadic-type algebra ${\cal A}_p$. In Section \ref{s:4}, we prove that ${\cal A}_p$ is not isomorphic to a set algebra. In this proof from the transposition operations we use only $P_{01}$. 
 
In Section \ref{s:5}, we show that if $p$ is large enough compared to $n$, namely if $p>2^{\alpha!n}+1$, then $n$-generated subalgebras of ${\cal A}_p$ are isomorphic to set algebras. In Section \ref{s:6}, we prove that both the cylindrification-free and the diagonal-free reducts of ${\cal A}_p$ are isomorphic to set algebras. The transposition-free reducts of ${\cal A}_p$ will be set algebras by their construction. 
These add together to a proof of Theorem \ref{t:2}. 

In Section \ref{s:eq}, we exhibit equations $e_p$ that witness that ${\cal A}_p$ is not representable, for odd prime powers $p\ge 3$. In more detail, we exhibit $e_p$ such that $e_p$ is true in all polyadic equality set algebras while it is not true in ${\cal A}_p$. Moreover, we show that $e_q$ is true in ${\cal A}_p$ whenever $q\ne p$. 

In Section \ref{s:var}, we investigate how far the equational theories of $Pse_{\alpha}$ and $Cs_{\alpha}$ are from each other, from a different point of view. Namely, we show that there are continuum many distinct equational theories between the two. Investigating the relevant portion of the lattice of varieties points to a new interesting problem about the complexity of axiom sets for the transposition operations, see Problem \ref{main-prb}.

Section \ref{const-s} contains some information about the antecedents of the construction used in the present paper, and Section \ref{inf-s} briefly describes the statuses of the analogous statements when $\alpha$ is infinite. Problems \ref{folax-p} and \ref{qpea-prob} in these sections ask  how far Theorem \ref{t:1}  can be generalized. Problem \ref{folax-p} asks whether ``equational axiom set" can be changed to ``first-order axiom set" in it and Problem \ref{qpea-prob} asks if the condition ``$\alpha<\omega$" can be relaxed in it.

The last section contains two applications of the results in the paper, one in logic and one in algebra. Theorem \ref{log-t} states that a complete proof system for usual finite-variable logic has to have a high syntactical complexity. In this result, the transposition operations do not occur explicitly.  Theorem~\ref{pea-p} provides a quite intuitive and transparent new equational axiom set for polyadic equality algebras of finite dimension, in terms of the transposition operations.

This paper uses several different branches of mathematics: combinatorics, algebra, logic, geometry. We try to be self-contained in the text and in footnotes we provide references to some background.

\section{The construction of ${\cal A}_p$ for $P_{01}$ and $\alpha = 3$}
\label{s:2}
For the sake of a more transparent explanation, first we consider the case where 
the dimension is 3 and just one transposition operation (namely $P_{01}$) is 
involved in the algebras in question.
Fixing the value of $p$ arbitrarily, first we construct a polyadic equality set algebra
$${\cal A}^s = \langle A, +, -, C_i, D_{ij}, P_{01} \rangle _{i,j < 3} $$
and then modify the effect of $P_{01}$ on a subset of $A$, to obtain an 
operation $P_{01}^*$ yielding an abstract algebra
$${\cal A}^* = \langle A, +, -, C_i, D_{ij}, P_{01}^* \rangle _{i,j < 3}$$
that will turn out to be 
not isomorphic to a set algebra. 
(We write ${\cal A}^*$ instead of ${\cal A}$, in order to express that just 
one polyadic operation is defined on the former, while the latter---to be 
constructed in Section \ref{s:3}---will be an algebra with ${\alpha}^{2}$ 
transposition operations.)

\bsk

\nin
{\bf Step 1.}\hsp 
Choose an odd prime power $p\ge 3$. 
Let  ${ U}$ be a set of cardinality $p^2+p-1$, and let it be partitioned 
into two disjoint sets, ${ U} = 
{ U}_0 \cup { U}_1$, with $|{ U}_0|=p^2$ and $|{ U}_1|=p-1$. 
The algebra ${\cal A}^s$ will have base set $ U$, i.e., the universe $A$ will consist of subsets of $ U \times  U \times 
 U$. 

\bsk

\nin
{\bf Step 2.}\hsp 
Let $AG(2,p)$ be the affine Galois plane with point set ${ U}_0$. 
Denoting by $\Lambda$ the set of lines of $AG(2,p)$, it is well-known%
\footnote{See, e.g., Construction 2.17 in Section VII.2.2 of \cite{HbComb}.} that 
$\Lambda$ can be partitioned into $p+1$ ``parallel classes'' ${\Lambda}_0, 
{\Lambda}_1, \ldots , {\Lambda}_p$, each ${\Lambda}_i$ consisting of $p$ 
mutually disjoint lines $L_{i,0}, \ldots , L_{i,p-1}$ (having $p$ points each).
Note that each parallel class $\Lambda_i$ defines an equivalence relation on ${ U}_0$.

\bsk

\nin
{\bf Step 3.}\hsp 
Decompose
 ${ U}_0 \times { U}_0 - Id_{{ U}_0}$ into $p+1$ symmetric and 
irreflexive binary relations $R_i$ ($i<p+1$) as follows:
$$R_i = \{ (u,v) \in { U}_0 \times { U}_0 - Id_{{U}_0} \mid (u,v) 
\in \bigcup _{0 \leq j<p} (L_{i,j} \times L_{i,j}) \} \vspace{-0.04in} . $$
Here, $Id_{{ U}_0}=\{ (u,u) : u\in { U}_0\}$ denotes the identity relation on ${ U}_0$.

Alternatively, viewing the elements of ${ U}_0$ as ordered pairs $(a,b)$ 
over the Galois field $GF(p)$ of order $p$, and choosing a bijection $h:p\to GF(p)$, we could have set
$$ R_i = \{ ((a,b),(a^\prime,b^\prime)) \mid a,a^\prime,b,b^\prime \in GF(p), \,
a\ne a^\prime , \, b^\prime - b = h(i)(a^\prime - a) \} $$
for $0 \leq i < p$, and
$$R_p = \{((a,b),(a,b^\prime)) \mid a,b,b^\prime \in GF(p), \, b \neq 
b^\prime \}.$$
One can observe that the decompositions of ${ U}_0 \times { U}_0 - Id_{{ U}_0}$ 
obtained in these two ways are in fact isomorphic.

\bsk

\nin
{\bf Step 4.}\hsp 
We partition $R_0 \times { U}_1$ 
into%
\footnote{$S\times X$ where $S$ is a binary relation and $X$ is a set denotes the ternary relation $\{ s : (s_0,s_1)\in S\mbox{ and }s_2\in X\}$. In the following, we will often use such suggestive notation without explicitly defining it.}
 $p-1$ mutually disjoint relations 
$Q_0, Q_1,$ $\ldots,Q_{p-2}$ in such a way that the following properties of the 
cylindrifications and the transposition operation $P_{01}$ be valid for 
every $ k <  p-1$.
  \begin{itemize}
 	\item[$(0)$] $ \quad C_0 (Q_k) \; \: =\ \  { U} \times { U}_0 
 	\times { U}_1$,
 	\item[$(1)$] $ \quad C_1 (Q_k) \; \: =\ \  { U}_0 \times { U} 
 	\times { U}_1$,
 	\item[$(2)$] $ \quad C_2 (Q_k) \; \: =\ \  R_0 \times { U}$,
 	\item[$(3)$] $ \quad P_{01} (Q_k)\  =\ \  Q_{p-2-k}$.
 \end{itemize} 
Let us recall that
 $p>2$ was chosen to be odd,
 hence $p-1$ is even, and 
therefore $P_{01}$ will be a one-to-one mapping from the set $\{Q_k \mid  k < p-1 \}$ onto itself, with no fixed element.

\begin{lemma}\label{q1-lem}
The relations\/ $Q_k$ with properties\/ $(0)$ through\/ $(3)$ 
exist.
\end{lemma}
\begin{proof}
Since the reflexive closure of $R_0$ is an equivalence relation on ${ U}_0 
\times { U}_0$ with equivalence classes $L_{0,j}$ $(j < p+1)$, it 
suffices to show that the required properties can be satisfied within each 
$L_{0,j} \times L_{0,j} \times { U}_1$. 

Let $L_{0,j} = \{v_i \mid 0 \leq i < p \}$. A natural way of partitioning 
$L_{0,j} \times L_{0,j} - Id_{L_{0,j}}$ into $p-1$ classes is to define 
$$ S_i = \{(v_k, v_{i+k+1}) \mid 0 \leq k < p \} $$
for $0 \leq i < p-1$, where subscript addition is taken modulo $p$. 
Note that $(a, b) \in S_i$ implies $(b, a) \in S_{p-2-i}$.

Let ${ U}_1 = \{z_i \mid 0 \leq i < p-1\}$. 
Putting $p^{\prime} = (p-1)/2$, we now consider a {\em bipartite graph}%
\footnote{For bipartite graph, perfect matching and other standard terms in graph theory see any textbook, e.g., \cite{graph}.} 
${\cal B} = (V, E)$ with vertex classes $V_1 = \{S_i \mid 0 \leq i < p-1 \}$ 
and $V_2 = { U}_1$ (and then $V = V_1 \cup V_2$) and with edge set
$$ E = \{ (S_i, z_l) \mid 0 \leq i, l < p^{\prime} \} \cup \{ (S_i, z_l) \mid 
p^{\prime} \leq i, l < 2 p^{\prime} \}. $$
Hence, every vertex of ${\cal B}$ has degree $p^{\prime}$, and ${\cal B}$ has 
two connected components.
We partition the edge set $E$ of ${\cal B}$ into $p^{\prime}$ ``perfect 
matchings''
$$ E_k = \{ (S_i, z_{i+k \: ({\rm mod \:} p^{\prime})}) \mid 0 \leq i < 
p^{\prime} \} \cup \{ (S_{p^{\prime}+i}, z_{p^{\prime}+(i+k \: ({\rm mod \:} 
p^{\prime}))}) \mid 0 \leq i < p^{\prime} \}. $$
(The edges within each $E_k$ are mutually disjoint.) 
Moreover, for $p^{\prime} \leq k < 2 p^{\prime}$, we set 
$$ E_k = \{ (S_i, z_l) \mid (S_i, z_{p-2-l}) \in E_{p-2-k} \}. $$
The pairs in the $E_k$ for $k \geq p^{\prime}$ are not edges in ${\cal B}$; 
instead, they decompose the ``bipartite complement'' of ${\cal B}$. 
Using the edge decomposition of ${\cal B}$ and its ``complementary'' 
collections of pairs, we finally define 
$$ Q_k = \{ (a, b, z_l) \mid (a, b) \in S_i, \, (S_i, z_l) \in E_k \}. $$
This last step of the construction ensures that $(3)$ is satisfied. 
It is also readily seen that 
$R_0 \times { U}_1 = \bigcup_{0 \leq k < p-1} Q_k$. 
Since \vspace{0.02in} $R_0 |_{L_{0,j}} = \bigcup_{0 \leq i < p-1} S_i$, 
and each $S_i$ is involved in each $E_k$, $(2)$ also holds. 
Last, $(0)$ and $(1)$ follow from the fact that the domain of $ S_i $ is $ L_{0,j}$ for every $i < p-1$. 
\end{proof}

\bsk

\nin
{\bf Step 5.}\hsp 
Let ${\cal A}^s = \langle A, +, -, C_i^U, D_{ij}^U, P_{01} \rangle _{i,j < 3}$ be the 
algebra 
where $+$ and $-$ are the Boolean 
operations (union and complementation with respect to $U$) 
and $A$ is generated by the set
$$ G = \{ Q_k \mid 0 \leq k < p-1 \} \cup \{R_i \times { U}_1 \mid 1 \leq i 
< p+1 \}
$$
which means that there is no proper subset of $A$ that contains $G$ and is closed under all operations of ${\cal A}^s$.
Then ${\cal A}^s$ is a set algebra, by definition.

\bsk

\nin
{\bf Step 6.}\hsp 
Finally, in order to obtain the abstract algebra 
$${\cal A}^* = \langle A, +, -, C_i, D_{ij}, P_{01}^* \rangle _{i,j < 3},$$ 
we replace equation $(3)$ of Step 4 \vspace{0.06in} by the requirement of
 \begin{itemize}
   \item[$(3^*)$] $ \quad P_{01}^* (Q_k) = Q_k$ for  \vspace{0.06in} every 
$k < p-1$,
 \end{itemize}
and modify the effect of $P_{01}$ on those parts of $A$ where necessary. 
In more detail: Every element of $A$ is the sum of the atoms below it, since $\langle A, +, - \rangle$ is a finite Boolean algebra.
Lemma \ref{2nd} below states that the $Q_k$ are atoms. We extend $P^*_{01}$ to $A$ by defining $P^*_{01}(a)=P_{01}(a)$ for the other atoms of ${\cal A}^s$ and requiring $P^*_{01}$ to be additive.

\begin{lemma}   \label{2nd}
The relations\/ $Q_k$ are atoms in\/ ${\cal A}^s$.
\end{lemma}

\begin{proof}  We are going to define a partition $At$ of  ${ U}\times { U}\times { U}$ such that 
 $ G\subseteq At\subseteq A$. I.e., the generator elements, members of $G$, are blocks in the equivalence relation defined by the partition, and all blocks in the equivalence relation are members of $A$. Then we show that the set $A'\subseteq A$ of sums of elements of $At$ is closed under the operations of ${\cal A}^s$. This will prove the lemma, as follows. The atoms of $A'$ are exactly the elements of $At$ by its definition, so each $Q_k$ is an atom in $A'$ by $Q_k\in G\subseteq At$. However, $A'=A$  since $A$ is generated by $G\subseteq At$, hence each $Q_k$ is an atom in $A$ as was to be shown.

We begin to define the partition $At$.  Let 
$${\cal R}=\{ R_i \mid i< p+1\}\quad\mbox{and}\quad{\cal D}=\{ Id_{{ U}_0}, Id_{{ U}_1}, Di_{{ U}_1}, { U}_0\times{ U}_1, { U}_1\times{ U}_0 \} .$$
Above, $Di_{{ U}_1}=\{ (u,v) : u,v\in { U}_1,\ u\ne v\}$ denotes the diversity relation on ${ U}_1$.
Then ${\cal R}\cup{\cal D}$ is a set of binary  relations on ${ U}$.  The elements of $At$, besides the $Q_k$, will be ternary relations on ${ U}$ defined by the use of ${\cal R}\cup{\cal D}$. Let $I=\{(0,1), (0,2), (1,2)\}$. For a function $t:I\to({\cal R}\cup{\cal D})$ let the ternary relation specified by $t$ be defined as
$$ a(t) = \{ s\in { U}\times { U}\times { U} \mid (s_i,s_j)\in t(i,j)\mbox{ for all }(i,j)\in I \} .$$
For example, $R_i\times { U}_1 = a(t_i)$ where $t_i(0,1)=R_i, t_i(0,2)=t_i(1,2)={ U}_0\times{ U}_1$. Let $F$ be the set of all functions from $I$ to ${\cal R}\cup{\cal D}$, except for the above defined $t_0$. We define $At$ as
$$ At = \{ a(t) \mid t\in F\} \cup \{ Q_k \mid  k < p-1\} $$
and let $A' = \{\sum X \mid X\subseteq At\}$. It is easy to see that the $a(t)$'s for distinct $t$'s are disjoint
(they may also be empty) and $\sum At={ U}\times { U}\times { U}$, thus $At$ is indeed a partition of the unit of $A$. Obviously, $G\subseteq At$. We note that $a(t_0)\in A'$ since $a(t_0)=R_0\times U_1=\sum \{ Q_k\mid k<p-1\}\in A'$. 

Now we show that $At\subseteq A$. It is enough to show that $a(t)\in A$ for every $t\in F$. We have already seen $R_i\times U_1\in A$. Thus $R_i\times U=C_2(R_i\times U_1)\in A$, and similarly $U_0\times U\times U$, $U\times U_0\times U$, $U\times U\times U_1$ are all in $A$, by applying cylindrifications. Further, $U_1\times U\times U=C_2(D_{02}\cap(U\times U\times U_1))$ and  $U\times U_1\times U=C_2(D_{12}\cap(U\times U\times U_1))$. From these, one can readily see that $R\times U\in A$ for all $R\in{\cal R}\cup{\cal D}$. Finally, it can be checked that
$$a(t) = (t(0,1)\times U)\cap C_1(D_{12}\cap (t(0,2)\times U))\cap C_0(D_{01}\cap C_1(D_{12}\cap (t(1,2)\times U))).$$

Next we show that $A'$ is closed under the operations of ${\cal A}^s$. By definition, it is closed under the Boolean operations $+,-$. For $i<j<3$ we have $D_{ij}=D_{ji}=\sum \{ a(t) \mid t(i,j)=Id_{U_0} \mbox{ or } t(i,j)=Id_{U_1}\}\in A'$ while $D_{ii}=U\times U\times U\in A'$ for all $i<3$. Thus, the diagonal constants are in $A'$. To show closure under $C_i$, it is enough to show that $C_i(a)\in A'$ for all $a\in At$, because $C_i$ is additive and by the definition of $A'$. Assume first $i=0$ and $a(t)\ne\emptyset$, then $$C_0(a(t))=\sum\{ a(t') \mid t'\in F\cup\{ t_0\}, t'(1,2)=t(1,2)\}\in A'$$ while 
$$C_0(Q_k)=U\times U_0\times U_1=\sum\{ a(t) \mid t\in F\cup\{ t_0\}, t(1,2)=U_0\times U_1\}\in A'.$$
The case $i\ne 0$ is similar, except that
$$C_2(Q_k) = R_0\times U = \sum \{ a(t) \mid t\in F\cup\{ t_0\}, t(0,1)=R_0\} .$$

\noindent In checking the $\supseteq$ part of the equation concerning $C_0(a(t))$, one can use that for all distinct $i,j,k<p+1$ and for all $(v,w)\in R_k$  there is $u$ such that $(u,v)\in R_i$ and $(u,w)\in R_j$. This is true by the construction of the $R_i, i<p+1$.

Finally, to show that $A'$ is closed under $P_{01}$, it is enough to show that $At$ is closed under $P_{01}$. Indeed, $\{ Q_k \mid k < p-1\}$ is closed under $P_{01}$ by property (3) in the definition of the $Q_k$ and
$P_{01}(a(t)) = a(t')$ where $t'(0,2)=t(1,2)$, $t'(1,2)=t(0,2)$ and $t'(0,1)=t(0,1)$ when $t(0,1)\notin\{ U_0\times U_1, U_1\times U_0\}$, and $t'(0,1)=U_0\times U_1$ if $t(0,1)=U_1\times U_0$, $t'(0,1)=U_1\times U_0$ if $t(0,1)=U_0\times U_1$. 
\end{proof}

\section{Larger $\alpha$ and more $P_{ij}$s}
\label{s:3}

Here we construct the algebras ${\cal A}_p$, that will satisfy all 
requirements of Theorem \ref{t:2}, in an analogous way as we have built up 
${\cal A}^*$ in the previous section.
Symbols such as ${\cal A}^s, G, {\cal R}, At, A'$ defined in the previous section will be redefined here, with related but slightly different meanings. 
Choose any odd prime power $p\ge 3$. Let  
${ U}$ be a set of cardinality $p^2+(\alpha-2)(p-1)$, and let it be partitioned into $\alpha-1$ disjoint sets ${ U}_k$ with $|{ U}_0|=p^2$ and $|{ U}_k|=p-1$ for $1\le k<\alpha-1$. 
The algebra ${\cal A}^s$ will have base set ${ U}= \bigcup\{ { U}_k \mid k<\alpha-1\}$. 

\bsk

\nin
Partition ${ U}_0 \times { U}_0 - Id_{{ U}_0}$ into $p+1$ symmetric and 
irreflexive binary relations $R_i$ as in Section \ref{s:2}.  
The important thing is, as we shall see in Section \ref{s:4}, that the $E_i=R_i\cup Id_{{ U}_0}$ are $p+1$ equivalence relations on ${ U}_0$ such that $E_i\cap E_j=Id_{{ U}_0}$ and $\bigcup\{ E_i \mid i < p+1\}=E_i\circ E_j={ U}_0\times { U}_0$ for distinct $i,j$, where 
$E_i\circ E_j = \{ (a,b) \mid (a,c)\in E_j\mbox{ and } (c,b)\in E_i\mbox{ for some }c\}$.%
\footnote{Relation composition is sometimes defined in different order and is called relative product. See, e.g., \cite[p.29]{HeMoT:II}, \cite{hh}. We use $\circ$ in this order to be compatible with function composition that we will use later in the paper.}
\bsk

\nin
Let $T={ U}_1\times\dots\times { U}_{\alpha-2}$.  Now, $T$ is an $(\alpha-2)$-place relation, and $R_0\times T$ is an $\alpha$-place relation.
We partition $R_0 \times T$ into $p-1$ mutually disjoint relations 
$Q_0, Q_1,\ldots,Q_{p-2}$ in such a way that the following properties of the 
cylindrifications and the transposition operation $P_{01}$ are valid for 
every \vspace{0.06in} $k < p-1$.
\begin{itemize}
	\item[{\bf (q0)}] $ \quad C_i (Q_k) \; \: = C_i(R_0 \times T)$,\quad for $i<\alpha $,
	\item[{\bf (q1)}] $ \quad P_{01} (Q_k) = Q_{p-2-k}$.
\end{itemize}

\begin{lemma}\label{q2-lem}
	The relations\/ $Q_k$ with properties\/ $(q0)$ and\/ $(q1)$ 
	exist.
\end{lemma}

\begin{proof} Take any system $S_k\subseteq R_0\times{ U}_1$, $k<p-1$ that exists according to Lemma \ref{q1-lem}. (These $S_k$ were called $Q_k$ in Lemma \ref{q1-lem}.) We will replace ${ U}_1$ with $T$ in it. Recall that $|{ U}_1|=p-1$. Take any partition $T_i, i<p-1$ of $T$ into $p-1$ subsets with the following property:
\begin{itemize}
	\item[{\bf (q2)}] $ \quad T \; \: \subseteq \; \: C_j (T_i)$,\quad for all $j<\alpha-2$ and $i<p-1$.
\end{itemize}
There is such a partition $T_i$. 
Indeed, let $f:{ U}\to P=\{ 0,1,\dots,p-2\}$ be such that $f$ maps ${ U}_j$ bijectively to $P$, for all $1\le j<\alpha-1$, and let $T_i=\{ s\in T \mid \sum \{ f(s_j) \mid j<\alpha-2\}=i\}$, where $\sum$ is meant modulo $p-1$. 
Let $s\in T$, $i< p-1$ and $j<\alpha-2$. Let $a\in{ U}_{j+1}$ be such that $f(a)=i-\sum\{f(s_l) \mid l\ne j\}$. Then $s(j/a)\in T_i$, where  $s(j/a)$ denotes the sequence that differs from $s$ only at $j$, and at $j$ it is $a$. Thus $T\subseteq C_j(T_i)$.

We define $Q_k$ by replacing $s_2$ in $s\in S_k$ with $T_{f(s_2)}$, that is, $Q_k=\{ s\in R_0\times T : \langle s_2,\dots, s_{\alpha-1}\rangle\in T_i\mbox{ and }\langle s_0,s_1,a\rangle\in S_k\mbox{ with }f(a)=i\}$.
Now, $Q_k$, $k<p-1$ is a partition of $R_0\times T$ since $S_k$, $k<p-1$ is a partition of $R_0\times{ U}_1$ and $T_i$, $i<p-1$ is a partition of $T$ and $|{ U}_1|=p-1$. Item (q0) holds for the $Q_k$ since their analogous (0),(1),(2) and (q2) hold for the $S_k$, and (q1) holds for the $Q_k$ by (3).	
\end{proof}

Let ${\cal A}^s = \langle A, +, -, C_i, D_{ij}, P_{ij} \rangle _{i,j < \alpha}$ be the 
algebra with base set ${ U}$ where $+$ and $-$ are the Boolean 
set operations (union and complement) in the set algebra $\langle A, +, - \rangle$, 
and $A$ is generated by the set
$$ G = \{ Q_k \mid k < p-1 \} \cup \{R_i \times T \mid 1 \leq i 
< p+1 \}
$$
in ${\cal A}^s$. 
Then ${\cal A}^s$ is a polyadic equality set algebra, by its definition.
\bsk

In order to be able to use a simpler language for talking about sequences of transposition operations in the rest of the paper, we introduce the following notation.  
We will use 
the fact 
that an ordinal is the set of smaller ordinals. Let $\tau:\alpha\to\alpha$ be a function. In the paper, we consider an $\alpha$-sequence $s=\langle s_0,s_1,\dots,s_{\alpha-1}\rangle$ to be a function mapping $\alpha$ to $U$, thus $s\circ\tau$ also maps $\alpha$ to $U$, where $\circ$ denotes the usual composition of functions; and $s\circ\tau=\langle s_{\tau0},s_{\tau1},\dots,s_{\tau(\alpha-1)}\rangle$ is the sequence $s$  rearranged along $\tau$.
Let $P(\alpha)$ denote the set of permutations of $\alpha$. For $\tau\in P(\alpha)$ and $X\subseteq U^{\alpha}$ we define
$$ S_{\tau}X = \{ s\in U^{\alpha} \mid s\circ\tau\in X\}.$$
The following are easy to check for $i,j<\alpha$ and $X\subseteq U^{\alpha}$.

\begin{description}  
	\item[(S1)] $S_{[i,j]}(X) = P_{ij}(X) $.
	\item[(S2)] $S_{\tau}S_{\sigma}(X)=S_{\tau\circ\sigma}(X)$.
	\item[(S3)] The $P_{ij}$, for $i,j<\alpha$  satisfy the polyadic equations (P1)--(P8) introduced in Section \ref{s:1}. 
\end{description}
In the rest of the paper, we will extensively use the above properties of $S_{\tau}$. Note that $S_{Id}X=X$ follows from (S1), (S3) and (P7).

 Let  $${\cal R}=\{ R_i\times T \mid i < p+1\}$$
and let ${\cal B}$ be the subalgebra of ${\cal A}^s$ generated by ${\cal R}$. First we show that the elements of ${\cal R}$ are atoms in ${\cal B}$. Recall that the relations $R_i$ are determined by parallel lines in an affine plane. This is how we defined them%
\footnote{If one wants to rely only on their properties mentioned at the beginning of this section, then one can use Lyndon's theorem stating that all such systems of relations come from affine planes, see \cite[Theorem 1]{Lyn61}.}
 in Section~\ref{s:2}. 
It is known that any two distinct points in a line can be taken to any other distinct two on a parallel line, in an affine plane, by a dilatation, which is a permutation of the points taking parallel lines to parallel ones. Such a dilatation is a permutation of ${ U}_0$ that takes all $R_i$ to themselves. Let us call this property 2-homogeneity.%
\footnote{For more details, see, e.g., \cite{Gi03}.} 
Now, let $f$ be a permutation of ${ U}$ that is a dilatation on ${ U}_0$  and that takes $U_k$ bijectively to $U_k$ for $0<k<\alpha-1$. Such an $f$ induces a permutation on ${ U}^{\alpha}$ which leaves the elements of ${\cal R}$ fixed. 
Then, $f$ leaves all the elements of ${\cal B}$ fixed, because the operations of ${\cal A}$ are permutation-invariant.%
\footnote{Permutation-invariance is an important property of logical connectives, for its definition see, e.g., \cite{hh}.}
 This implies that $R_i\times T$ is an atom in ${\cal B}$ since any sequence in $R_i\times U$ can be taken to any one in it by a permutation of the above kind, by 2-homogeneity of the affine plane. 

Let $At{\cal B}$ denote the set of atoms of ${\cal B}$, we have just seen that ${\cal R}\subseteq At{\cal B}$.  Recall that $P(\alpha)$ denotes the set of permutations of $\alpha$, and let
$$ At = \{ S_{\tau}(Q_k) \mid \tau\in P(\alpha), k< p-1\}\cup At{\cal B} - \{ S_{\tau}(R_0\times T) \mid \tau\in P(\alpha)\}. $$

\begin{lemma}\label{at2-lem}
	$At$ is the set of atoms of ${\cal A}^s$.
\end{lemma}

\begin{proof} The proof is analogous to the proof of Lemma \ref{2nd}, but we do not have an analogous description of the atoms of ${\cal A}^s$.%
\footnote{We cannot have one for $\alpha\ge 5$ by \cite[proof of Theorem 12.37]{hh}.}
	
We begin with showing that $At\subseteq A$. Let $\tau\in P(\alpha)$. It is known that each permutation of $\alpha$ can be written as a composition of transpositions $[i,j]$, say $\tau=[i_1,j_1]\circ\dots\circ[i_r,j_r]$. Then $S_{\tau}(Q_k)=P_{i_1j_1}(\dots P_{i_rj_r}(Q_k)\dots)$, so $S_{\tau}(Q_k)\in A$ by $Q_k\in G\subseteq A$. Also, $At{\cal B}\subseteq B\subseteq A$ by ${\cal R}\subseteq A$. Here, $B$ denotes the universe of ${\cal B}$.

The elements of $At$ are disjoint from each other, and $\sum At={ U}^{\alpha}$, thus $At$ is a partition of the unit of ${\cal A}^s$.  Let $A' = \{\sum X \mid X\subseteq At\}$. We have seen that $At\subseteq A$, hence also $A'\subseteq A$. We claim that $A'=A$, this will imply that $At$ is the set of atoms of ${\cal A}^s$.

We show that $A'$ is closed under the operations of ${\cal A}^s$. By definition, $A'$ is closed under the Boolean operations $+,-$. 
Let $i,j<\alpha$. 
Concerning closure under $P_{ij}$, it is enough to show that $At$ is closed under $P_{ij}$ because the $P_{ij}$ are additive. Indeed, $\{ S_{\tau}Q_k \mid \tau\in P(\alpha), k< p-1\}$ is closed under $P_{ij}$ by $P_{ij}S_{\tau}a=S_{\tau'}a$  for all $a$ where $\tau'=[i,j]\circ \tau$. Also, $At{\cal B}- \{ S_{\tau}(R_0\times T) \mid \tau\in P(\alpha)\}$ is closed under $S_{\tau}$ because $B$ is closed under $S_{\tau}$ by ${\cal B}$  being a subalgebra of ${\cal A}^s$ and since the set of  omitted atoms is closed under all $S_{\tau}$.

Next we show $B\subseteq A'$. If $a\in At{\cal B}-At$ then $a$ is $S_{\tau}(R_0\times T)$ for some $\tau\in P(\alpha)$ and then $a\in A'$ by $R_0\times T=\sum_kQ_k$ and additivity of $S_{\tau}$. Thus, $At{\cal B}\subseteq A'$ by $At\subseteq A'$. Now, $B\subseteq A'$ follows from $A'$ being closed under sums and since each element of $B$ is a sum of atoms of ${\cal B}$ (by ${\cal B}$ being finite). 

This immediately implies that the diagonals $D_{ij}$ are in $A'$ since they are in $B$.
To show closure under $C_i$, it is enough to show that $C_i(a)\in B\subseteq A'$ for all $a\in At$, because $C_i$ is additive. Now, $C_iQ_k=C_i(R_0\times T)\in B$ by $R_0\times T\in B$. Then $C_iS_{\tau}Q_k=S_{\tau'}C_jQ_k$ for some $\tau'$ and $j$, so it is in $B$, too. Clearly, $C_ib\in B$ for $b\in At{\cal B}$.

We can now prove our claim that $A'=A$. We have seen that $A'\subseteq A$ and $A'$ is closed under the operations of ${\cal A}^s$. Since ${\cal A}^s$ is generated by $G\subseteq At\subseteq A'$, this implies that $A'=A$ and we are done, by $At$ being the set of atoms of $A'$.
\end{proof}

We want to define the operation $P^*_{01}$ of ${\cal A}_p$ so that $P^*_{01}(Q_k)=Q_k$ as in the case of $\alpha=3$, but now we have to define $P^*_{ij}(S_{\tau}Q_k)$ as well, for all $i,j<\alpha$ and $\tau\in P(\alpha)$. 
A small problem here is that the element $S_{\tau}Q_k\in A$ does not determine $\tau$ and $k$ uniquely since $S_{\tau}Q_k=S_{\tau\circ[0,1]}Q_{p-2-k}$. Luckily, these are the only coincidences:

\begin{lemma}\label{tau-l}  $S_{\tau}Q_k=S_{\sigma}Q_j$ implies $(\sigma,j)\in\{(\tau,k),(\tau\circ[0,1],p-2-k)\}$.
\end{lemma}
\begin{proof}
	Assume $S_{\tau}Q_k=S_{\sigma}Q_j$. We begin by showing that $\tau(i)=\sigma(i)$ for all $1<i<\alpha$. Indeed, $S_{\tau}Q_k\ne\emptyset$ by $Q_k\ne\emptyset$, so take any $s\in S_{\tau}Q_k$. Then $s\circ\tau\in Q_k$, therefore $s(\tau(i))\in U_{i-1}$ and $s(\tau(j))\notin U_{i-1}$ for all $j\ne i$, by $Q_k\subseteq U_0\times U_0\times T$. By $s\in S_{\tau}Q_k=S_{\sigma}Q_j$ and the analogous argument we get that $s(\sigma(i))\in U_{i-1}$. By $\tau,\sigma\in P(\alpha)$ we have that $\sigma(i)=\tau(j)$ for some $j$, so $s(\tau(j))\in U_{i-1}$, hence $\tau(i)=\tau(j)=\sigma(i)$. Thus $\tau(i)=\sigma(i)$ for all $1<i<\alpha$, so $\{\tau(0),\tau(1)\}=\{\sigma(0),\sigma(1)\}$, by $\tau,\sigma\in P(\alpha)$. If $\sigma(0)=\tau(0)$ then $\sigma=\tau$. If $\sigma(0)=\tau(1)$ then $\sigma=\tau\circ[0,1]$.  If $\sigma=\tau$ then $Q_j=Q_k$ since the $Q_k$'s are distinct and $S_{\tau}$ is a bijection on $A$. If $\sigma=\tau\circ[0,1]$ then $S_{\tau}Q_k=S_{\sigma}Q_j=S_{\tau}P_{01}Q_j$, from which we get $Q_k=P_{01}Q_j$, i.e., $Q_j=P_{01}Q_k=Q_{p-2-k}$ by $P_{01}\circ P_{01}=Id$ and property (q1) in the definition of the $Q_k$. 
\end{proof}

From $\tau$ and $\tau\circ[0,1]$ exactly one is monotonic on $0,1$ and this is how we will assign $(\sigma,j)$ to $S_{\tau}Q_k$.
For a permutation $\tau\in P(\alpha)$, we define $\tau^+$ as $\tau$ if $\tau(0)<\tau(1)$, and $\tau\circ[0,1]$ otherwise. Thus, $\tau^+$ agrees everywhere with $\tau$ except perhaps on $0,1$, and otherwise $\tau^+(0)<\tau^+(1)$. In particular,  $[0,1]^+=Id$. 
Then $\{ S_{\tau}Q_k \mid \tau\in P(\alpha), k<p-1\}=\{ S_{\tau^+}Q_k \mid \tau\in P(\alpha), k<p-1\}$.

For $i,j<\alpha$ we define, for $\tau\in P(\alpha)$ and $k<p-1$, 
\begin{description}
	\item[] $P^*_{ij}(S_{\tau^+}Q_k)=S_{([i,j]\circ\tau^+)^+}Q_k$,
	\item $P^*_{ij}a=P_{ij}a$ for all other $a\in At$,
	\item and we extend $P^*_{ij}$ to all elements of $A$ by  requiring it to be additive. 
\end{description}
We are ready to define ${\cal A}_p$ as
$${\cal A}_p = \langle A, +, -, C_i, D_{ij}, P^*_{ij}\rangle_{i,j<\alpha}\ . $$

\section{Nonrepresentability of ${\cal A}_p$} 
\label{s:4}

In this section, we prove that the algebras ${\cal A}_p$ constructed in Section \ref{s:3} are 
not isomorphic to set algebras. By what we said after Theorem \ref{t:2} in Section \ref{s:1}, this will imply that they are nonrepresentable.
Concrete equations true in $Pse_{\alpha}$ but not true in ${\cal A}_p$ will be exhibited in Section \ref{s:eq}. 
In the proof we will use only $P^*_{01}$, so in fact we will show that 
$${\cal A}_p^* = \langle A,+,-,C_i,D_{ij},P^*_{01}\rangle_{i,j<\alpha}$$
is not isomorphic to a set algebra. This will imply that ${\cal A}_p$ is not isomorphic to a set algebra, either.

Assume that $X,Y\subseteq U^{\alpha}$ are $\alpha$-place relations. We say that $X$ is {\it symmetric} when $P_{01}X=X$ and we say that $X\subseteq Y$ is a {\it big} subset of $Y$ when $Y\subseteq C_iX$ for all $i<\alpha$; the latter holds exactly when $C_iX=C_iY$ for all $i<\alpha$.

The proof of nonrepresentability of ${\cal A}_p^*$ hinges on the combinatorial fact that $R_0\times U^{\alpha-2}$ cannot be partitioned into $p-1$ big and symmetric relations. It is partitioned into $p-1$ big and nonsymmetric relations in the set algebra ${\cal A}^s$, but in ${\cal A}_p^*$ the modified operation $P_{01}^*$ ``states" that these big elements are symmetric (by $P_{01}^*(Q_k)=Q_k$).
We will see in Section \ref{s:5} that $R_0\times U^{\alpha-2}$ can be split into fewer symmetric big relations, the proof of representability of small subalgebras will be based on this fact.

We will use the following two lemmas.

\begin{lemma}   \label{l:x}
Let\/ $q\ne 0$ be a natural number, and suppose that\/ $S_0, \ldots, S_q$ are 
nontrivial equivalence relations on a set\/ $Z$, with the 
following properties: 
 \begin{itemize}
   \item $\bigcup_{i \leq q} S_i = Z \times Z$,
   \item $S_i \cap S_j = Id_Z$\quad for all\/ $i$ and\/ $j$, $i \neq j$, and
   \item $S_i \circ S_j= Z \times Z $\quad for\/
$i \neq j$.
 \end{itemize}
Then  each equivalence class of\/ $S_0$ has precisely\/ $q$ 
elements.
\end{lemma}
\begin{proof}
Note first that $S_0$ has at least two equivalence classes, for 
otherwise all of the $S_i$ but $S_0$ would be trivial, by $S_i \cap S_0 = Id_Z$, and $q\ne 0$. 
Let $X$ be an arbitrary equivalence class of $S_0$. Fix an arbitrary $y \in Z - X$ and consider the set $S^* = \{(x, y) \mid x \in X \}$. 
Certainly, $S_0 \cap S^* = \emptyset$. 
We claim that $|S_i \cap S^*| = 1$ for every $i$, $1 \leq i \leq q$. 
(Having verified this, the lemma will be proved, by $\bigcup S_i=Z\times Z$.)

Suppose that $|S_i \cap S^*| > 1$ for some $i$; say, $(x, y), (x^{\prime}, y) \in 
S_i$ $(x, x^{\prime} \in X, \, x \neq x^{\prime})$. 
Then, since $S_i$ is symmetric and transitive, we obtain that $(x, x^{\prime}) 
\in S_i$ also holds, yielding the contradiction $(x, x^{\prime}) \in S_0 \cap S_i 
\neq Id_Z$. 
Hence, it suffices to show that $S_i \cap S^* \neq \emptyset$ for every $i > 0$. 
Pick any $x \in X$. Then $(x,y)\in Z\times Z=S_i\circ S_0$, so
there is a $z \in Z$ with $(x, z) \in S_0$ and $(z, y) \in 
S_i$. 
Since $X$ is an equivalence class of $S_0$, it must be the case that $z \in X$. 
Thus, $(z, y) \in S_i \cap S^*$. 
\end{proof}

\begin{lemma}   \label{l:y}
Let $Z$ be finite. If there exist\/ $|Z|-1$ mutually disjoint, symmetric and irreflexive 
relations\/ $S_0, \ldots , S_{|Z|-2}$ each with domain $Z$, such that\/  
$\bigcup_{i < |Z|-1} S_i = Z \times Z - Id_Z$, then\/ $Z$ has an even 
number of elements.
\end{lemma}
\begin{proof}
Assign color $i$ $(0 \leq i \leq |Z|-2)$ to an unordered pair $\{ x, y \} 
\subseteq Z$ if $(x, y) \in S_i$. 
Since the $S_i$ are symmetric, and their (disjoint) union is the set of all 
unordered pairs in $Z$, each pair gets precisely one color. 
On one hand, each color occurs on the set $N(x) = \{\{ x,y\} \mid y \in Z - \{x\} 
\}$ for every fixed $x \in Z$ (for otherwise $x$ would not belong to the domain of $S_i$ for some $i$). 
On the other hand, $|N(x)| = |Z|-1$, therefore each of the $|Z|-1$ colors occurs 
{\em precisely}\/ once in $N(x)$. 
Consequently, the pairs $\{z, z^\prime\}$ of color $0$ are mutually disjoint and 
their union is $Z$, so that $|Z|$ is even, indeed. 
\end{proof}

\begin{proof}[Proof of nonrepresentability of ${\cal A}_p^*$.]
Assume that 
there is an isomorphism $h:{\cal A}_p^* \to h({\cal A}_p^*)$ where $h({\cal A}_p^*)$ is a set algebra. We derive a contradiction.

First, by using $h$ and $R_i$, $i<p+1$, we get binary relations $\bar{S_i}$, $i<p+1$ that satisfy the hypotheses of Lemma \ref{l:x}, as follows. 
Let $V$ denote the base set of $h({\cal A}_p^*)$ and let $m<p+1$. Then $h(C_1\dots C_{\alpha-1}(R_0\times T))=h(U_0\times U^{\alpha-1})=V_0\times V^{\alpha-1}$, for some set $V_0$, because $C_1\dots C_{\alpha-1}(R_0\times T)=a=C_1\dots C_{\alpha-1}a$, where
$a=U_0\times U^{\alpha-1}$. Then $h(U_0\times U_0\times U^{\alpha-2})=V_0\times V_0\times V^{\alpha-2}$ by $U_0\times U_0\times U^{\alpha-2}=a\cap C_0(D_{01}\cap a)$ and $V_0\times V_0\times V^{\alpha-1}=b\cap C_0(D_{01}\cap b)$ where $b=h(a)=V_0\times V^{\alpha-1}$.
Similarly,  $h(R_m\times U^{\alpha-2})=S_m\times V^{\alpha-2}$, for some binary relation $S_m$ on $V_0$. Moreover, 
$\bar{S}_m=S_m\cup Id_{V_0}$ is an equivalence relation on $V_0$, because $R_m\cup Id_{U_0}$ is an equivalence relation on $U_0$ and being an equivalence relation can be expressed in cylindric set algebras of dimension at least $3$ for binary relations, as follows.%
\footnote{For details see, e.g., the chapter on relation algebra reduct \cite[Chapter 5.3]{HeMoT:II}.} A binary relation $S$ on $V_0$ is symmetric iff 
${}_2s(0,1)(S\times V^{\alpha-2})=S\times V^{\alpha-2}$, and $S$ is transitive iff $(S\times V^{\alpha-2}) ; (S\times V^{\alpha-2})\subseteq S\times V^{\alpha-2}$  where for distinct $i,j,k$
\begin{description}
	\item{} ${}_ks(i,j)x=s^k_is^i_js^j_kx$\quad with\quad $s^i_jx=c_i(d_{ij}\cdot x)$\quad and
	\item{} $x;y = c_2(s^1_2x\cdot s^0_2y)$\quad for all $x,y$ and distinct $i,j,k<\alpha$.
\end{description}
The $S_m$'s are nonempty and disjoint from each other, since the $R_m$'s are, and also each $S_m$ is disjoint from $Id_{V_0}$ since $R_m$ is disjoint from $Id_{U_0}$. Further, $\bigcup S_m=V_0\times V_0 - Id_V$ by $\bigcup R_m=U_0\times U_0 - Id_U$, and similarly $S_m\circ S_k =V_0\times V_0 - (S_m\cup S_k\cup Id_V)$ for distinct $m,k$. Thus, $\bar{S}_m$, $m<p+1$ and $V_0$ satisfy the conditions of Lemma \ref{l:x}, hence $|W|=p$, for each equivalence class $W$ of $\bar{S}_0=S_0\cup Id_{V_0}$, by Lemma \ref{l:x}.

Now, we will show that $|W|$ is even, by using Lemma \ref{l:y} and the system $Q_k$, $k<p-1$ of $\alpha$-place relations together with $P^*_{01}$. 
To use Lemma \ref{l:y}, we define a system $Z_i$, $i<p-1$ of binary  relations on $W$ by using the images  $P_k=h(Q_k)$ of the relations $Q_k$, $k<p-1$. By $C_1\dots C_{\alpha-1}Q_0=U_0\times U^{\alpha-1}$ and $h$ being a homomorphism we have that $C_1\dots C_{\alpha-1}P_0=V_0\times V^{\alpha-1}$. Thus, by $ W\subseteq V_0$, there is $s\in P_0$ with $s_0\in W$. For an $\alpha$-sequence $s$ and elements $a,b$, let $s(0\slash a,1\slash b)$ be the sequence that agrees everywhere with $s$ except perhaps on $0,1$, and on $0,1$ it takes the values $a,b$ respectively. For all $i<p-1$, define
$$Z_i = \{(a,b)\in W\times W \mid s(0\slash a,1\slash b)\in P_i\} .$$
Let $u=s_0\in W$. For all $w\in W - \{ u\}$ we have that $(u,w)$ is in a unique $Z_i$ because of the following. For all $i<p-1$ there is $w_i\in W$ such that $(u,w_i)\in Z_i$, by $C_1Q_0=C_1Q_i$ and by $P_i\subseteq S_0\times V^{\alpha-2}$. Now, all $w_i$ and $u$ are distinct because $P_i\cap P_j=\emptyset$ and $P_i\cap Id_{V_0}=\emptyset$ for distinct $i,j$. By $|W|=p$ then we have that $W=\{ u, w_0, \dots, w_{p-2}\}$. By using $C_0Q_0=C_0Q_i$, by a similar argument we have that for all $w\in W$, $w\ne u$ and $i<p-1$ there is a unique $u_i$ such that $(u_i,w)\in Z_i$. Using $C_1Q_0=C_1Q_i$ again, we get that $\bigcup_{i<p-1}Z_i=W\times W - Id_W$ and the domain of $Z_i$ is $W$ for all $i<p-1$. The $Z_i$ are irreflexive since $R_0$ is irreflexive, and the $Z_i$ are symmetric by $P^*_{01}Q_i=Q_i$ and the definition of $Z_i$. By Lemma \ref{l:y} then $|W|$ is even.

However, we have seen that $|W|=p$ where $p\ge 3$ is odd,
and so we arrived at a contradiction.
By this, we have proved that ${\cal A}_p^*$ is not isomorphic to a set algebra. 
\end{proof}

\section{Representable subalgebras of ${\cal A}_p$}
\label{s:5}

In this section, we prove that the $n$-generated subalgebras of ${\cal A}_p$ are isomorphic to set algebras when  $p>2^{\alpha!n}+1$. 
This will imply that the $n$-generated subalgebras of ${\cal A}_p^*$ are also isomorphic to set algebras, because if $X\subseteq A$ then the subalgebra of ${\cal A}_p^*$ generated by $X$ is a subalgebra of the appropriate reduct of the subalgebra of ${\cal A}_p$ generated by $X$. If the latter is a set algebra then so is the former.

\bigskip

Let ${\cal A}={\cal A}_p=\langle A,+,-,C_i,D_{ij},P^*_{ij}\rangle_{i,j<\alpha}$ with $p$ as above, and let $X\subseteq A$ be arbitrary such that $|X|\le n$. We are going to show that the subalgebra of ${\cal A}$ generated by $X$ is representable.
The idea of the proof is to show that the subalgebra does not separate at least two $Q_k$'s, and when we ``split" $R_0\times T$ to only $p-2$ big parts, the so obtained $Q_k$'s can be chosen to be symmetric.

Let $X_1=\{ P^*_{i_1j_1}\dots P^*_{i_mj_m}x : x\in X, 0\le m, i_1,j_1,\dots,i_m,j_m\in\alpha\}$, let $Bg X_1$ denote the set Boolean-generated by $X_1$ in ${\cal A}$, that is, $Bg X_1$ is the smallest subset of $A$ containing $X_1$ and closed under the Boolean operations $+,-$ of ${\cal A}$, and finally let $X_2$ be the set of atoms of $Bg X_1$. Then $|X_2| \le 2^{|X_1|}$. 
We are going to show that $|X_1|\le \alpha!n$. This will imply $$|X_2|< p-1$$ by our assumption on $p$. Recall that $\tau^+$ was defined in Section \ref{s:2}.

\begin{lemma}\label{plem} 
	Let $m\ge0, i,j,i_1,j_1,\dots,i_m,j_m\in\alpha$ and $\sigma=[i_1,j_j]\circ\dots\circ[i_m,j_m]$. For all $\tau\in P(\alpha)$, $k<p-1$ and $a\in At$ not of form $S_{\tau}Q_k$ for any $\tau$ and $k$ the following statements hold.
	\begin{description}
		\item[(i)] $P^*_{i_1j_1}\dots P^*_{i_mj_m}S_{\tau^+}Q_k = S_{(\sigma\circ \tau^+)^+}Q_k$.
				\item[(ii)] $P^*_{i_1j_1}\dots P^*_{i_mj_m}a = S_{\sigma}a$.
				\item[(iii)] $P_{ij}^*:A\to A$ is a Boolean automorphism of ${\cal A}$, i.e., it is a bijection that respects the operations $+,-$ of ${\cal A}$.
	\end{description}
\end{lemma}

\begin{proof}
To prove (i), we proceed by induction on $m$. The statement holds for $m=0$ since in this case $\sigma=Id$, and $\tau^+=\tau^{++}$. Assume that the statement is true for $m$, we show that it is true for $m+1$. Let $\delta=[i_{m+1},j_{m+1}]$ and $\sigma_m=[i_1,j_1]\circ\dots\circ[i_m,j_m]$, $\sigma_{m+1}=\sigma_m\circ\delta$.
\begin{description}
	\item[] $P^*_{i_1j_1}\dots P^*_{i_{m+1}j_{m+1}}(S_{\tau^+}Q_k)  = $
	\item[] $P^*_{i_1j_1}\dots P^*_{i_mj_m}(P^*_{i_{m+1}j_{m+1}}(S_{\tau^+}Q_k)) =$
	\item[] $P^*_{i_1j_1}\dots P^*_{i_mj_m}(S_{(\delta\circ\tau^+)^+}Q_k)=$
	 \item[] $S_{(\sigma_m\circ(\delta\circ\tau^+)^+)^+}Q_k = 
	 S_{(\sigma_m\circ\delta\circ\tau^+)^+}Q_k = S_{(\sigma_{m+1}\circ\tau^+)^+}Q_k$.
\end{description}
In the penultimate equality we used validity of $(\rho\circ\eta^+)^+ = (\rho\circ\eta)^+$ for all permutations $\rho,\eta$. Indeed, $\rho\circ\eta^+$ and $\rho\circ\eta$ agree everywhere except perhaps on $0$ and $1$ and then $(\rho\circ\eta^+)^+$ and $(\rho\circ\eta)^+$ agree everywhere since they agree on $0$ and $1$, too.
We get (ii) immediately by the definition of $P^*_{ij}$. %

To prove (iii), it is enough to show that $P_{ij}^*:At\to At$ is a bijection, since $P_{ij}^*$ is additive by its definition and $A$ is finite. Now, $P_{ij}^*(a)\in At$ for $a\in At$, by the definition of $P_{ij}^*$. Since $At$ is finite, it remains to show that $P_{ij}^*$ is injective. If $a\in At{\cal B}$ and $b\in At$, $b\ne a$ then $P_{ij}^*(a)$ and $P_{ij}^*(b)$ are distinct by the definition of $P_{ij}^*$ and since $\{ S_{\tau}Q_k \mid \tau\in P(\alpha), k< p-1\}$ is disjoint from $At{\cal B} - \{ S_{\tau}(R_0\times T) \mid \tau\in P(\alpha)\}$. So, let $a=S_{\tau}Q_k$, $b=S_{\sigma}Q_m$ and assume that$P_{ij}^*(S_{\tau}Q_k)=P_{ij}^*(S_{\sigma}Q_m)$. We want to show $a=b$.  We may assume $\tau(0)<\tau(1)$ and $\sigma(0)<\sigma(1)$, i.e., $\tau=\tau^+$ and $\sigma=\sigma^+$ by $S_{\tau}Q_\ell=S_{\tau\circ[0,1]}Q_{p-\ell-2}$ for all $\ell<p-1$. We may assume also $i<j$ by $P_{ij}^*=P_{ji}^*$.

By the definition of $P_{ij}^*$, we have $S_{f^+}Q_k=S_{g^+}Q_j$, where $f=[i,j]\circ\tau$ and $g=[i,j]\circ\sigma$. By Lemma \ref{tau-l}, we have either $f^+=g^+$ or $f^+=g^+\circ[0,1]$. We cannot have the second case, because $f^+(0)< f^+(1)$ and $g^+(0)< g^+(1)$, by the definition of $f^+, g^+$. Then, by Lemma \ref{tau-l} we have $f^+=g^+$ and $j=k$. Hence, for showing $a=b$, it is enough to show $\tau=\sigma$, and for this, it is enough to show $f=g$.

We have $f^+=g^+$.  So $f\in\{ g, g\circ[0,1]\}$. Suppose for contradiction that $f=g\circ[0,1]$. That is, $[i,j]\circ\tau=[i,j]\circ\sigma\circ[0,1]$. Then $\tau=\sigma\circ[0,1]$. So $\tau(0)=(\sigma\circ[0,1])(0)=\sigma(1)$  and $\tau(1)=(\sigma\circ[0,1])(1)=\sigma(0)$, which is incompatible with our assumption that $\tau(0)<\tau(1)$ and $\sigma(0)<\sigma(1)$. Hence $f=g$ as reqired.
\end{proof}

By Lemma \ref{plem} we get that the value of $P^*_{i_1j_1}\dots P^*_{i_mj_m}x$ depends only on $\sigma=[i_1,j_1]\circ\dots\circ[i_m,j_m]$,  since the $P^*_{ij}$'s are additive and $x$ is a sum of atoms. Since $\sigma$ is a permutation of $\alpha$, there are at most $\alpha!$ values in $X_1$ for all $x\in X$, hence we get $|X_1|\le\alpha!n$, and from this we obtain $|X_2|<p-1$ as mentioned already.

Before proceeding, it is convenient to introduce a notation. For  $\sigma\in P(\alpha)$ we define $S^*_{\sigma} : A\to A$ by 
$$S^*_{\sigma}x = P^*_{i_1j_1}\dots P^*_{i_mj_m}x\quad\mbox{ where }\sigma=[i_1,j_1]\circ\dots\circ [i_m,j_m] .$$
This definition is sound by Lemma \ref{plem}.

\bigskip
We are going to show that there are distinct $k,m< p-1$ such that no element of $X$ separates $S^*_{\sigma}Q_k$ from $S^*_{\sigma}Q_m$, for all $\sigma\in P(\alpha)$. For all $i<p-1$ there is a unique $a_i\in X_2$ such that $Q_i\le a_i$, since the $Q_i$ are atoms in ${\cal A}$ and $X_2$ is a partition of the unit of $A$.  There are distinct $k,m<p-1$ such that $a_k=a_m$ since $|X_2|<p-1$. Thus, no element of $X_2$ separates $Q_k$ from $Q_m$, i.e., $Q_k\le x$ iff $Q_m\le x$ for all $x\in X_2$. By its definition, $X_1$ is closed under the $S^*_{\sigma}$'s, i.e., $S^*_{\sigma}x\in X_1$ for all $x\in X_1$ and $\sigma\in P(\alpha)$. The $S^*_{\sigma}$'s are Boolean automorphisms in ${\cal A}$ by Lemma \ref{plem}(iii) and the definition of $S_{\sigma}^*$. Therefore $Bg X_1$ and $X_2$ are also closed under the $S^*_{\sigma}$'s. This implies that no element of $X_2$ separates $S^*_{\sigma}Q_k$ from $S^*_{\sigma}Q_m$, either, for all $\sigma\in P(\alpha)$. By $X\subseteq Bg X_1 = Bg X_2$ then we get that no element of $X$ separates $S^*_{\sigma}Q_k$ from $S^*_{\sigma}Q_m$, for all $\sigma\in P(\alpha)$. 
Define
$$C = \{ a\in A \mid (S^*_{\sigma}Q_k\le a\mbox{ iff }S^*_{\sigma}Q_m\le a)\quad\mbox{ for all $\sigma\in P(\alpha)$} \} .$$
We have just seen that $X\subseteq C$. We are going to show that $C$ is closed under the operations of ${\cal A}$, hence it is the universe of a subalgebra ${\cal C}$ of ${\cal A}$. By $X\subseteq C$ then the subalgebra generated by $X$ is a subalgebra of ${\cal C}$. Hence it will be enough to prove that ${\cal C}$ is 
isomorphic to a set algebra.

Clearly, $C$ is closed under the Boolean operations $+,-$, and the diagonal constants $D_{ij}\in C$ for all $i,j<\alpha$.
Further, $C$ is closed under $P^*_{ij}$ because of the following. Assume $a\in C$. Then $S^*_{\sigma}Q_k\le P^*_{ij}a$ iff $P^*_{ij}S^*_{\sigma}Q_k\le P^*_{ij}P^*_{ij}a = a$ iff $S^*_{[i,j]\circ\sigma}Q_k\le a$ iff $S^*_{[i,j]\circ\sigma}Q_m\le a$ iff $S^*_{\sigma}Q_m\le P^*_{ij}a$. It remains to show that $C$ is closed under the cylindrification operations $C_i$. Assume that $a\in C$ and $S^*_{\sigma}Q_k\le C_ia$. Then $C_iS^*_{\sigma}Q_k\le C_ia$. However, $C_iS^*_{\sigma}Q_k=C_iS^*_{\sigma}Q_m$ by condition (q0) in the choice of the $Q_k$, by Lemma \ref{plem} and the properties of the set operations $S_{\tau}$. Hence $S^*_{\sigma}Q_m\le C_ia$ and we are done with showing that ${\cal C}$ is a subalgebra of ${\cal A}$. 

It is easy to check that the atoms of ${\cal C}$ are
those of ${\cal A}$, except that both $S_{\tau}Q_k$ and $S_{\tau}Q_m$ are replaced with $S_{\tau}(Q_k+Q_m)$, for all $\tau\in P(\alpha)$. 
This implies that also  both $S_{\tau}Q_{p-2-k}$ and $S_{\tau}Q_{p-2-m}$ are replaced with $S_{\tau}(Q_{p-2-k}+Q_{p-2-m})$, because $Q_{p-2-k}=P_{01}Q_k$ and the same holds for $Q_m$.
Let $Q_i'=Q_i$ where $i\notin\{ k,m,p-2-k, p-2-m\}$, let $Q_k'=Q_m'=Q_k+Q_m$ and let $Q_{p-2-k}'=Q_{p-2-m}'=Q_{p-2-k}+Q_{p-2-m}$. 
Then the set of atoms of ${\cal C}$ is
$$ At_C = \{ S_{\tau^+}Q_i' \mid \tau\in P(\alpha), i<p-1\}\cup At{\cal B}-\{ S_{\tau}(R_0\times T) \mid \tau\in P(\alpha)\} $$
and $C=\{\sum X : X\subseteq At_C\}$. Note that $|\{ Q_i' \mid i<p-1\}|\in \{ p-2, p-3\}$ depending on whether $Q_m=Q_{p-2-k}$ or not.

We are going to exhibit a polyadic equality set algebra ${\cal D}$ that is isomorphic to ${\cal C}$. We use $U, U_i, R_i$ as in the algebra ${\cal A}$. The algebra ${\cal D}$ will be like ${\cal A}$ except that, instead of the $Q_i$ we use a different partition of $R_0\times T$. 
As one may recall, the idea of the proof of nonrepresentability of ${\cal A}$ was that $R_0\times T$ cannot be partitioned into $p-1$ symmetric big relations. The next lemma shows that such a partition into $p-2$ parts is possible.

\begin{lemma}\label{klem}
There is a partition of $R_0 \times T$ into $p-2$ mutually disjoint relations 
$K_0, K_1,\ldots,K_{p-3}$ in such a way that the following properties of the 
cylindrifications and the polyadic operation $P_{01}$ are valid for 
every \vspace{0.06in} $j < p-2$.
\begin{itemize}
	\item[{\bf (q0)}] $ \quad C_i (K_j) = C_i(R_0 \times T)$, \quad for $i<\alpha $,
	\item[{\bf (k1)}] $ \quad P_{01} (K_j) = K_j$.
\end{itemize}
\end{lemma}

\begin{proof} It is enough to show existence of the $K_i$ on $W\times W\times T$ for a set $W$ of cardinality $p$, because each equivalence class of $R_0\cup Id$ has cardinality $p$. Let $W=\{ w_0, \dots, w_{p-1}\}$.
Take a proper edge-coloring%
	\footnote{An edge-coloring is called proper when no adjacent edges have the same color.} of the complete graph on vertex set $W_0=\{ w_0,\dots,w_{p-2}\}$, with $p-2$ colors; it is known that such a coloring exists, as $p-1$ is even.%
	\footnote{According to Lucas \cite{Luc--92}, the first such coloring was constructed by Walecki near the end of the 19th century. As the  $i$-th edge-class ($i=0,1,\dots,p-3$) one can take $\{(w_{i+j},w_{i-j}) \mid 1\leq j\leq (p-3)/2\} \cup \{ (w_{p-2},w_{i})\}$, with subscript addition modulo $p-2$. 
		For generalization to complete uniform hypergraphs, see the milestone paper of Baranyai \cite{Bar75}.} This coloring gives a partition of $W_0\times W_0 - Id$ into $p-2$ symmetric and irreflexive binary relations $\gamma_0, \dots,\gamma_{p-3}$ each with domain $W_0$. Define $\rho_i=\gamma_i\cup\{ (w_{p-1},w_i), ( w_i,w_{p-1})\}$ for $i<p-3$, and $\rho_{p-3}=\gamma_{p-3}\cup(\{ w_{p-1}\}\times\{ w_{p-3},w_{p-2}\})\cup(\{ w_{p-3},w_{p-2}\}\times\{ w_{p-1}\}) $ . 
Then $\rho_0,\dots,\rho_{p-3}$ is a partition of $W\times W - Id$ into $p-2$ symmetric, irreflexive relations each with domain $W$. Take the system $T_i$, $i<p-1$ from the proof of Lemma \ref{q2-lem}, and let $J_i=T_i$ for $i<p-3$ while $J_{p-3}=T_{p-3}\cup T_{p-2}$. Define now $K_i^W=\bigcup\{ \rho_j\times J_k \mid i=j+k (\mbox{mod $p-2$})\}$. The relations $K_i = \bigcup\{ K_i^W \mid W\mbox{ is a block of }R_0\}$ satisfy  the requirements.
\end{proof}

We are ready to define ${\cal D}$. Assume that $|\{ Q_i' \mid i<p-1\}|=p-2$. The case when this number is $p-3$ will be completely analogous, we will omit it.
Let
$$ At_D = \{ S_{\tau}K_i \mid \tau\in P(\alpha), i<p-2\}\cup At{\cal B}-\{ S_{\tau}(R_0\times T) \mid \tau\in P(\alpha)\} $$
and let
$$ {\cal D}=\langle D,+,-,C_i,D_{ij},P_{ij}\rangle_{i,j<\alpha}\qquad \mbox{ where } D = \{ \sum X \mid X\subseteq At_D\} .$$
One can see, as in the proof of Lemma \ref{at2-lem}, that $D$ is closed under the operations of ${\cal D}$, hence ${\cal D}\in Pse_{\alpha}$.

It  remains to exhibit an isomorphism between ${\cal C}$ and ${\cal D}$.  The following bijection between $C$ and $D$ suggests itself.  Take a bijection $h:\{ Q_i' : i<p-1\}\to \{ K_i : i<p-2\}$. There is such a bijection because there are $p-2$ many $Q_i'$'s and there are $p-2$ many $K_i$'s. Extend $h$ to $h:C \to D$ as follows: $h(S_{\tau^+}Q_i')=S_{\tau^+}h(Q_i')$ for $i<p-2$ and $\tau\in P(\alpha)$, let $h(a)=a$ for the other atoms in $At_C$, and let $h(\sum X)=\sum h(X)$ for $X\subseteq At_C$. It is easy to check that $h$ is a bijection between $C$ and $D$, and it is a homomorphism with respect to $+,-,C_i, D_{ij}$. To see $h(P^*_{ij}(x))=P_{ij}h(x)$ for all $x\in C$, it is enough to check this for all $x$ of form $S_{\tau^+}Q_i'$. We can use Lemma \ref{plem} here, as follows. 
Notice first that $S_{\tau^+}x=S_{\tau}x$ when $x$ is symmetric, for all $\tau\in P(\alpha)$.  
Then
\begin{description}
\item{} $h(P^*_{ij}S_{\tau^+}Q_q')\quad \ \ =$ \qquad by Lemma \ref{plem} 
\item $h(S_{([i,j]\circ\tau^+)^+}Q_q') \ =$ \qquad by definition of $h$ 
\item $S_{([i,j]\circ\tau^+)^+}h(Q_q') \ = $ \qquad since $h(Q_q')$ is symmetric
\item $S_{[i,j]\circ\tau^+}h(Q_q')\quad \ =$ \qquad since $S_{[i,j]} = P_{ij}$ and $S_{\sigma\circ\delta}x=S_{\sigma}S_{\delta}x$
\item $P_{ij}S_{\tau^+}h(Q_q')\quad \ \ =$ \qquad by definition of $h$
\item $P_{ij}h(S_{\tau^+}Q_q')$.
\end{description}
We have shown that ${\cal C}$ is isomorphic to the polyadic equality set algebra ${\cal D}$.

\section{Representable reducts of ${\cal A}_p$}
\label{s:6}
First we show that the cylindrification-free reduct  ${\cal A}^c=\langle A,+,-,D_{ij},P^*_{ij}\rangle_{i,j<\alpha}$ is 
isomorphic to a set algebra.  From now on, we sometimes say ``representable" to mean ``isomorphic to a set algebra". 
This time it is enough to choose a partition of $R_0\times T$ into $p-1$ symmetric but not necessarily big relations. This is not hard to do. Indeed,
let $W=\{ w_0, \dots, w_{p-1}\}$ be an arbitrary set of cardinality $p$. For $i<p-2$, let $S_i=\{ (w_0,w_{i+1}),(w_{i+1},w_0)\}$, and let $S_{p-2}=W\times W - Id -\bigcup\{ S_i \mid i<p-2\}$. For $i<p-1$ let  $H^W_i=S_i\times T$, and let $H_i=\bigcup \{ H_i^W \mid W\mbox{ is a block of }R_0\}$. Now, $H_i$, $i<p-1$ is a partition of $R_0\times T$ into $p-1$ symmetric relations. 

If we replace $Q_k$ in the construction of ${\cal A}$ with $H_k$ then $P^*_{ij}$ will agree with $P_{ij}$ because the $H_k$ are symmetric, and all the other operations except for the cylindrifications remain the same. This is the idea of showing that ${\cal A}^c$ is representable. 
In more detail: Recall that ${\cal A}$ denotes ${\cal A}_p$. In defining the set algebra ${\cal D}$ we use $U, U_i, R_i$ and ${\cal B}$ as in the definition of ${\cal A}$. Let $At_D=\{ S_{\tau}H_k \mid \tau\in P(\alpha), k<p-1\}\cup At{\cal B} - \{ S_{\tau}(R_0\times T) \mid \tau\in P(\alpha)\}$, let $D=\{ \sum X \mid X\subseteq At_D\}$ and let ${\cal D}=\langle D,+,-,D_{ij}^U,P_{ij}\rangle_{i,j<\alpha}$. 
Then ${\cal D}$ is a set algebra because it is a subalgebra of the cylindrification-free reduct of the polyadic set algebra $\langle{\cal P}(U^{\alpha}),+,-,C_i^U,D_{ij}^U,P_{ij}\rangle_{i,j<\alpha}$,  where ${\cal P}(U^{\alpha})$ denotes the set of all subsets of $U^{\alpha}$.
Let $h:At\to At_D$ be defined as $h(S_{\tau^+}Q_k)=S_{\tau^+}H_k$ and $h(a)=a$ for the other atoms. Then it is not hard to show that the additive extension of $h$ to $A$ is an isomorphism between ${\cal A}^c$ and ${\cal D}$. In showing isomorphism with respect to the transposition operations $P_{ij}^*$ and $P_{ij}$, we can use Lemma \ref{plem} as we did in the previous section.
\bigskip

The construction of a set algebra isomorphic to the diagonal-free reduct ${\cal A}^d=\langle A,+,-,C_i,P^*_{ij}\rangle_{i,j<\alpha}$ is a bit more involved. The idea is that we double each element of $U$, and in the bigger set we can find a partition of $(R_0\times T)'$ into $p-1$ symmetric big relations. We now elaborate this idea.

Let $U$ be the base set of ${\cal A}$ and let $f:U\to f(U)$ be a bijection between $U$ and a set $f(U)$ disjoint from $U$. For all $u\in U$ let $u'=\{ u,f(u)\}$ and let $U'=\bigcup\{ u' : u\in U\}=U\cup f(U)$. For all $a\subseteq U^{\alpha}$, let $a'\subseteq U'^{\alpha}$ be defined as
$$ a'=\bigcup\{ s_0'\times\dots \times s_{\alpha-1}' \mid \langle s_0,\dots,s_{\alpha-1}\rangle\in a\}.$$
Now, the function $F(a)=a'$ respects all operations of ${\cal A}$, except for the diagonals. Indeed, $F(a+b)=F(a)+F(b)$ etc, but $F(D_{ij}^U)\ne D_{ij}^{U'}$, because, e.g., $\langle u,f(u),u,\dots,u\rangle\in F(D_{01}^U) - D_{01}^{U'}$.

We turn to partitioning $(R_0\times T)'$ into $p-1$ symmetric and big relations. Let $m=(p-1)\slash 2$. Then $K_i=Q_i\cup Q_{p-2-i}$ for $i<m$ is a partitioning of $R_0\times T$ into $m$ symmetric and big relations. We now partition each $K_i$ into two symmetric and big relations. 
For $i<m$ let
\begin{description}
	\item{} $S^0 = (U\times U) \cup (f(U)\times f(U))$,    
	\item{} $S^1 = (U\times f(U)) \cup (f(U)\times U)$, 
	\item{} $K_i^0 = \{ \langle s_0,s_1,\dots,s_{\alpha-1}\rangle\in K_i' \mid (s_0,s_1)\in S^0\}$, 
	\item{} $K_i^1 = \{ \langle s_0,s_1,\dots,s_{\alpha-1}\rangle\in K_i' \mid (s_0,s_1)\in S^1\}$. 
\end{description}
Let $H_0, H_1,\dots, H_{p-2}$ be an enumeration of $K_i^0, K_i^1$, $i<m$. Then $H_0,\dots, H_{p-2}$ is a partition of $(R_0\times T)'$ into $p-1=2m$ symmetric big relations.  Define
\begin{description}
	\item{} $At_D = 
	\{ S_{\tau}H_k \mid \tau\in P(\alpha), k<p-1\}\cup F(At{\cal B}-\{ S_{\tau}(R_0\times T) \mid \tau\in P(\alpha)\} )$,
	\item ${\cal D} = \langle D,+,-,C_i^{U'}, P_{ij}\rangle_{i,j<\alpha}\qquad \mbox{ where } \quad D=\{\sum X \mid X\subseteq At_D\}$.
\end{description}
As before, ${\cal D}$ is a set algebra and we define $h:At\to At_D$ by $h(S_{\tau^+}Q_k)=S_{\tau^+}H_k$, and $h(a)=a'$ for the other elements of $At$. Then we extend $h$ to $A$ by requiring it to be additive. Now, it is straightforward to check that $h$ is an isomorphism between ${\cal A}^d$ and ${\cal D}$.

\section{Witness equations}
\label{s:eq}

In this section, we exhibit an equation $e_p$ that holds in $Pse_{\alpha}$ but does not hold in ${\cal A}_p^*$, for each odd prime power $p$.
The equation $e_p$ expresses that the properties of $R_i$, $i<p+1$  and $Q_k, k<p-1$ imply that $Q_0$ is not symmetric. We also show that $e_q$ holds in ${\cal A}_p^*$ if $q\ne p$.

We gather the properties of the $R_i$ and $Q_k$ that were used in showing that ${\cal A}_p^*$ is not representable.
The set of equations below expresses that $R_i$, $i<p+1$ is a partition of $U_0\times U_0-Id$ such that the $R_i$ are symmetric and transitive with domain $U_0$ and $R_i\circ R_j=U_0\times U_0 - (R_i\cup R_j\cup Id)$ if $i\ne j$.%
\footnote{In relation-algebraic terminology, the above equations express that the $x_i$ are the diversity atoms of a Lyndon relation algebra. For more details on Lyndon algebras see, e.g.,  \cite[Sec.\ 4.5]{hh}, \cite{Lyn61}, \cite[Ch.\ 6.30--32]{Madd}.} Recall the terms ${}_2s(i,j)x$ and $x;y$ introduced towards the end of Section \ref{s:4}.
\begin{description}
	\item{} $\sum\{ x_i : i< p+1\}=c_1x_0\cdot c_0x_0 - d_{01},\quad x_i\cdot x_j=0$,
	\item{} $x_i=c_2\dots c_{\alpha-1}x_i = {}_2s(0,1)x_i,\quad x_i;x_i\le x_i+d_{01},\quad  c_1x_i=c_1x_0,\quad c_0c_1x_i=1$,
	\item{} $ x_i;x_j = \sum\{ x_k : k< p+1, k\ne i,j\}$, \qquad  for $i,j< p+1, i\ne j$.
\end{description}
\noindent
The set of equations below expresses the properties of $Q_k$, $k<p-1$ that were used in showing that ${\cal A}_p^*$ is not representable: 
\begin{description}
	\item{} $c_1\dots c_{\alpha-1}y_0=c_1x_0,\quad y_i\le x_0, \quad y_i\cdot y_j=0$,
	\item{} $c_0y_i=c_0y_0,\quad c_1y_i=c_1y_0, \quad p_{01}y_0=y_0$,\qquad for $i,j< p-1$, $i\ne j$.
\end{description}
Let $E_p(x_0,\dots,x_p,y_0,\dots,y_{p-2})$ 
denote the union of the two finite sets of equations displayed above. In Section \ref{s:4}, in the proof of nonrepresentability of ${\cal A}_p^*$,  we showed that $E_p$ is true in ${\cal A}_p^*$ when the variables $x_0,\dots,x_p,y_0,\dots,y_{p-2}$ are evaluated to $R_0\times U^{\alpha-2},\dots, R_p\times U^{\alpha-2},Q_0,\dots,Q_{p-2}$ respectively, and we also showed, by using Lemmas \ref{l:x} and \ref{l:y}, that $E_p$ cannot be true under any evaluation of the variables in a set algebra: 
\begin{equation}
\tag{*} {\cal A}_p^*\models E_p(R_0\times U^{\alpha-2},\dots, Q_{p-2})\quad\mbox{while}\quad Pse_{\alpha}\models\forall x_0\dots\forall y_{p-2} \lnot\wedge E_p.
\end{equation}

\nin 
We are going to show that the set $E_p$ can be replaced with a single equation $e_p$ such that
$${\cal A}^*_p\not\models e_p\quad\mbox{ while }\quad Pse_{\alpha}\models e_p .$$
Indeed, by $\alpha$ being finite, there is a so-called switching term for set algebras.%
\footnote{In technical terms, this means that $Pse_{\alpha}$ is a discriminator class. For more details, see \cite[Sec.\ 9]{BuSa}, \cite[Sec.\ 2.6.4]{hh} or \cite[Sec.\ 2.7]{AGyNS}.} Namely, in set algebras we have 
$$ x=0 \mbox{ iff } c_0c_1\dots c_{\alpha-1}x=0\qquad\mbox{ and }\qquad x\ne 0 \mbox{ iff } c_0c_1\dots c_{\alpha-1}x=1 .$$
This switching term, i.e., $c_0c_1\dots c_{\alpha-1}x$, is good also for ${\cal A}_p$ because its transpo\-sition-free reduct is a set algebra. 
We now show that by the use of this switching term we can construct an equation $e_p$ such that ${\cal A}_p\models E_p[v]$ iff ${\cal A}_p\not\models e_p[v]$, for all evaluations $v:\{ x_0,\dots,y_{p-2}\}\to A$ of the variables, and the same holds for any member of $Pse_{\alpha}$ in place of ${\cal A}_p$.
Let $\oplus$ denote symmetric difference and let $\Pi$ denote the group-use of the Boolean $\cdot$, i.e., 
$x\cdot y=-(-x+-y)$, $x\oplus y=(x\cdot -y)+(-x\cdot y)$, and $\Pi\{ z_0,\dots,z_k\}=z_0\cdot...\cdot z_k$. With this notation, we define $e_p$ as
$$ \Pi\{-c_0\dots c_{\alpha-1}(\tau\oplus\sigma) \mid \tau=\sigma\in E_p\} = 0  .$$

\begin{lemma}\label{l:n}
${\cal A}^*_p\not\models e_p$  and $Pse_{\alpha}\models e_p$  for all odd prime powers $p$.
\end{lemma}

\begin{proof} Let ${\cal A}$ be any algebra whose cylindric reduct is a set algebra. Then the above displayed switching term works in it. Let $v:\{ x_0,\dots,y_{p-2}\}\to A$ be an evaluation of the variables occurring in $e_p$ to $A$.  It is enough to show, by (*), that
$${\cal A}\models e_p[v]\quad\mbox{ iff }\quad {\cal A}\not\models E_p[v] .$$ 
Indeed, ${\cal A}\models e_p[v]$ iff ${\cal A}\models -c_0\dots c_{\alpha-1}(\tau\oplus\sigma)=0[v]$ for some $\tau=\sigma\in E_p$, because the value of each term in the product is $0$ or $1$, by the property of the switching term. Now,  ${\cal A}\models -c_0\dots c_{\alpha-1}(\tau\oplus\sigma)=0[v]$ iff ${\cal A}\models c_0\dots c_{\alpha-1}(\tau\oplus\sigma)=1[v]$, iff by the properties of the switching term,  ${\cal A}\models\tau\oplus\sigma\ne 0[v]$, iff ${\cal A}\models\tau\ne\sigma[v]$.	
\end{proof}
\noindent
Lemma \ref{l:r} below implies that $$E = \{ e_p \mid \mbox{ $p$ is an odd prime power}\}$$ is an independent set of equations, in the sense that%
\footnote{For this notion of independence see, e.g., \cite[0.1.29]{HeMoT:II}.}
no member of $E$ follows from the set of all other members of $E$. 

\begin{lemma}\label{l:r}
	${\cal A}^*_p\models e_q$ for all odd prime powers $q$ different from $p$.
\end{lemma}	

\begin{proof} Let $E_q^0$ denote the set of equations in $E_q$ that concern only $x_0,\dots,x_q$. It is enough to show, by the construction of $e_q$, that there are no $S_0,\dots,S_q\in A$ that satisfy $E_q^0$ when $q\ne p$. We can work in the set algebra ${\cal A}_p^s$	because no transposition operation occurs in $E_q^0$. Assume that $S_0,\dots, S_q$ satisfy $E_q^0$; we derive a contradiction. There are binary relations $S_i'$ such that $S_i=S_i'\times U^{\alpha-2}$, by 
$x_i=c_2\dots c_{\alpha-1}x_i\in E_q^0$. Let $S_i^+=S_i'\cup Id_V$ where $V$ is the domain of the relation $S_0'$. Then $S_i^+$, $i\le q$ satisfy the hypotheses of Lemma \ref{l:x}, by the rest of the equations in $E_q^0$. Thus, each $S_i^+$ is an equivalence relation on $V$ such that each equivalence class of $S_i^+$ has cardinality $q$, by Lemma \ref{l:x}. Also, the $S_i'$'s are pairwise disjoint.

Recall that ${\cal B}$ denotes the subalgebra of ${\cal A}_p^s$ which is generated by the $R_i\times T$, $i\le p$. Now, $S_i\in B$ by $S_i\in A$, $S_i=C_2\dots C_{\alpha-1}S_i$, and by the construction of ${\cal A}_p^s$. Therefore, all permutations of $U$ that leave all the $R_i\times T$ unchanged, leave the $S_i$ also unchanged.  Thus, if $(u,v)\in S_i'$ then $U_k\times U_j\subseteq S_i^+$ if $u\in U_k$, $v\in U_j$, and $k,j\ne 0$; because all permutations of the $U_l$ for $l\ne 0$ extend to permutations of $U$ that leave all the $R_i\times T$ unchanged. 
Hence, each equivalence class $X$ of $S_i^+$ contains $U_k$ if $u\in U_k\cap X$, $k\ne 0$. 
Thus, if the domain of $S_i^+$ is disjoint from $U_0$  then each equivalence class of $S_i^+$ has even cardinality, because $|U_l|=p-1$ for all $l\ne 0$. Since $q$ is not even, the domain of $S_i^+$ intersects $U_0$, for each $i\le q$. 

We show that $S_i'$ contains some $R_j$. 
Indeed, assume $(u,w)\in S_i'$, $u\in U_0$. Then $(v,w)\in S_i'$ for some $v\in U_0$, $v\ne u$ since there is a permutation of $U$ that leaves $U-U_0$ pointwise fixed, takes $u\in U_0$ to some $v\in U_0$, $v\ne u$ and leaves $S_i'$ fixed. Thus, $(u,v)\in S_i'$ by $S_i^+$ being symmetric and transitive.
Then $(u,v)\in R_j$ for some $j\le p$, because $\sum R_j=U_0\times U_0 - Id$, by the construction of ${\cal A}_p^s$. Recall from the proof of Lemma \ref{at2-lem} that the equivalence classes of $R_j\cup Id_{U_0}$ are parallel lines in an affine geometry, and any two distinct points on a line can be taken to any other pair of distinct points on a parallel line by a permutation of $U_0$ that fixes all the $R_j$. This implies that $R_j\subseteq S_i'$ if $R_j\cap S_i'\ne\emptyset$. On the other hand, $S_i'$ cannot contain $R_j\cup R_k$ for $j\ne k$, because of the following. By the equations in $E_q^0$, $S_i^+$ is transitive and we have $R_j\circ R_k=U_0\times U_0 - (R_j\cup R_k \cup Id)$, so $S_i'$ contains all the $R_k$ if it contains more than one of them. This is a contradiction, since the $S_j'$'s are disjoint and all of them have to contain some $R_j$. We have seen that each $S_i'$ must contain a single $R_j$, which implies that $q\le p$. 

On the other hand, $\sum S_i'\supseteq U_0\times U_0 - Id = \sum R_j$ because of the following. Equations in $E_q^0$ yield $\sum S_i=C_1S_0\cap C_0S_0 - D_{01}$, $C_1S_0=V\times U^{\alpha-1}$, and $C_0S_0=U\times V\times U^{\alpha-2}$. 
So $\sum S_i=V\times V\times U^{\alpha-2}-D_{01}$. We deduce that  (*) $\sum S_i'=V\times V-Id$. We saw that there is $j\le p$ with  $R_j\subseteq S_i'$. So, the domain of $R_j$ is a subset of the domain of $S_i'$, i.e., $U_0\subseteq V$. Now (*) yields $\sum S_i'\supseteq U_0\times U_0 - Id$, as required.
Thus $\sum S_i'\supseteq U_0\times U_0 - Id = \sum R_j$, which implies $q\ge p$. Hence $q=p$, but we chose $q\ne p$. 	 
\end{proof}

\section{On the gap between representable cylindric and polyadic algebras}
\label{s:var}

In this section, we prove some results on the lattice of varieties between 
the varieties of all representable cylindric and all representable polyadic equality algebras, and we pose a new problem suggested by these results. A class of algebras is called a {\it variety}, or an equationally definable class, if there is a set of equations such that the class of algebras consists of all algebras in which this set of equations is true.  

Recall that a polyadic-type algebra is called {\it representable polyadic algebra} iff the equational theory of $Pse_{\alpha}$ is true in it. Likewise, let us call the cylindric reduct of a polyadic-type algebra {\it representable cylindric algebra} iff the equational theory of $Cs_{\alpha}$ is true in it.
Let $RCA_{\alpha}$ and $RPEA_{\alpha}$ denote the classes of all representable cylindric and polyadic algebras, respectively. These are varieties, by their definitions. 

Let $RCPEA_{\alpha}$ denote the class of all polyadic-type algebras whose cylindric (i.e., transposition-free) reduct is representable and in which the equations (P1)--(P8) hold. Then $RCPEA_{\alpha}$ is also a variety, since it is the class of all algebras in which (P1)--(P8) together with all equations true in $Cs_{\alpha}$ are true. Theorem \ref{pea-p} in Section \ref{log-s} implies that $RCPEA_{\alpha}$ is the class of all polyadic equality algebras whose cylindric reduct is representable.

The present paper concerns with the gap between $RCA_{\alpha}$ and $RPEA_{\alpha}$. The question in \cite[p.\ 348]{Jo} asks whether the finite set of polyadic equality axioms bridges this gap, and Problem 1 of \cite{Jo} asks whether $RPEA_{\alpha}$ is not finitely axiomatizable over $RCA_{\alpha}$. The first question is equivalent to asking whether the varieties $RCPEA_{\alpha}$ and $RPEA_{\alpha}$ coincide. Theorem \ref{big-t} below states that not only they do not coincide, but the gap between them is as big as it can be, in the sense that there are continuum many varieties between them.

\begin{theorem}\label{big-t}
	There are continuum many varieties between $RCPEA_{\alpha}$ and $RPEA_{\alpha}$, for each $3\le\alpha<\omega$.
\end{theorem}

\begin{proof} The set $E$ of equations defined in the previous section is denumerably infinite, since there are infinitely many odd prime powers. Thus, the power set of $E$ has cardinality continuum. For $X\subseteq E$ let $V(X)$ be the class of those $RCPEA_{\alpha}$ in which $X$ is true. Then $V(X)$ is a variety between $RCPEA_{\alpha}$ and $RPEA_{\alpha}$. Assume that $X,Y\subseteq E$ and $X\ne Y$. Then there is $e_p\in X$ such that $e_p\notin Y$, or the other way round, we may assume the first case. By $e_p\notin Y$ and Lemma \ref{l:r} we have that ${\cal A}_p\in V(Y)$. By $e_p\in X$ and Lemma \ref{l:n} we have that ${\cal A}_p\notin V(X)$. Thus $V(X)\ne V(Y)$ for all distinct $X,Y\subseteq E$.	
\end{proof}

Figure \ref{var-fig} depicts some facts about the lattice of varieties%
\footnote{Investigation of the lattice of subvarieties of a given variety is common in algebraic logic (and in universal algebra).  See, e.g., \cite{Blok} and \cite[Sec.\ 4.1]{HeMoT:II}.}
 between $RCA_{\alpha}$ and $RPEA_{\alpha}$. For notational convenience, instead of the classes of algebras, it depicts their equational theories. Note that smaller classes of algebras have bigger equational theories, so we have the equational theory of $RPEA_{\alpha}$ at the top in the picture. 

The figure represents a partially ordered set of some equational theories between those of $RCPEA_{\alpha}$ and $RPEA_{\alpha}$.  Two nodes are connected with a sequence of ``upward going" lines if and only if the lower node is a subset of the upper one, except that we do not know whether the top theory $\overline{RPEA}$ is a subset of the one below it. This is indicated by a question mark in the figure. Thus, the indicated theories are all distinct, except perhaps for the top two ones; see Problem \ref{main-prb}.

Before giving evidence for the above statements, we define the equational theories indicated in the figure.
We proceed from the bottom of the figure towards its top.

Fix $3\le\alpha<\omega$ and let $PE$ denote the set of polyadic equations (P1)--(P8) introduced in Section \ref{s:1}.

$\overline{RCA}$ denotes the equational theory of representable cylindric algebras enriched with arbitrary  transposition functions $p_{ij}$ for $i,j<\alpha$.

$\overline{RCPEA}$ denotes the equational theory of the class of polyadic equality algebras whose cylindric reducts are representable. This is the same as the equational closure of $\overline{RCA}\cup PE$.

$\overline{RCPEA+X}$ denotes the equational closure of $\overline{RCPEA}\cup X$  for $X\subseteq E$, where $E$ is the set of equations defined in Section \ref{s:eq}.


$\overline{RCPEA+P_{01}}$ denotes the equational theory of the set algebras in $RCPEA_{\alpha}$ in which also $p_{01}$ is the ``real" transposition operation $P_{01}$.

Finally, $\overline{RPEA}$ denotes the  equational theory of $RPEA_{\alpha}$.

\bsk

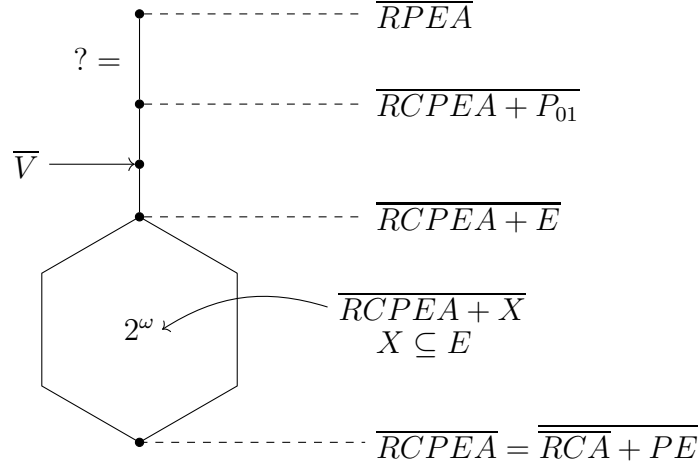
\begin{figure}[h]
	\begin{center}
		
		\begin{tikzpicture}
			\foreach \i in {1,2,3,4,5,6}
			{
				\coordinate (node-\i) at ({60*\i-30}:1.5cm) {};
				\node at (node-\i) {};
			}
			\node at (node-5) [fill,circle, draw, inner sep=1.1pt] {};
			\draw (node-1) -- (node-2) -- (node-3) -- (node-4) -- (node-5) -- (node-6) -- (node-1);
			
			\node (h0) at ({60*2-30}:1.5cm) [fill, circle, draw, inner sep=1.1pt] {}; 
			\node (h1) at (0,3) [fill, circle, draw, inner sep=1.1pt] {}; 
			\node (h2) at (0,4.2) [fill, circle, draw, inner sep=1.1pt] {}; 
			
			\draw (h0) --  (h1) -- (h2);
			\coordinate (node-5-pair) at (3,-1.5) {};
			\coordinate (h0-pair) at (3,1.5) {};
			\coordinate (h1-pair) at (3,3) {};
			\coordinate (h2-pair) at (3,4.2) {};
			
			\draw[dashed] (node-5) -- (node-5-pair);
			\draw[dashed] (h0) -- (h0-pair);
			\draw[dashed] (h1) -- (h1-pair);
			\draw[dashed] (h2) -- (h2-pair);
			
			\node[right] at (h0-pair) {$\overline{RCPEA+E}$};
			\node[right] at (h2-pair) {$\overline{RPEA}$};
			\node[right] at (h1-pair) {$\overline{RCPEA+P_{01}}$};
			\node[right] at (node-5-pair) {$\overline{RCPEA} = \overline{\overline{RCA}+PE}$};
			\coordinate (hmid-pair) at (-1.2,2.2) {};
			
			\node[left] at (-0.1,3.6) {$?=$};
			
			\node at (0,0) {$2^{\omega}$};
			\node[right] at (0.5+2,0.3+0) {$\overline{RCPEA+X}$};
			\node[right] at (0.5+2.5,0.3+-0.5) {$X\subseteq E$};
			
			\draw[->]  (0.5+2,0.3+0) to [bend right=25]  (0.3,0);
			
		\end{tikzpicture}
		\caption{Some equational theories between $\overline{RCPEA}$ and $\overline{RPEA}$}
		\label{var-fig}
	\end{center}
\end{figure}

We begin to prove the statements about these theories that we made above. We proceed from the bottom towards the top.

The equivalence  $\overline{RCPEA+X}\subseteq\overline{RCPEA+Y}$ iff $X\subseteq Y$ is shown in the proof of Theorem \ref{big-t}. The inclusion $\overline{RCPEA+E} \subseteq \overline{RCPEA+P_{01}}$ holds because of the transposition operations only $p_{01}$ occurs in $E$.

\begin{claim}\label{v-claim}
	$\overline{RCPEA+E} \ne \overline{RCPEA+P_{01}}$.
\end{claim}
\begin{proof}
	We exhibit an algebra ${\cal A}\models\overline{RCPEA+E}$ such that ${\cal A}\not\models\overline{RCPEA+P_{01}}$.
	It is a variation of ${\cal A}_p$, the only difference is that we use different tools for making sure that $|U_0|$ is odd up to isomorphism. Namely, $U_0$ has cardinality 5 in ${\cal A}$ and, instead of a Lyndon algebra, we use the diversity atoms of the so-called pentagonal algebra to ensure this. For more details see Section \ref{const-s}, the exact definition of ${\cal A}$ can be found in \cite{APEA87}. In this algebra, $p_{01}$ is not representable for the same reason why ${\cal A}_p^*$ is not representable, and an equation analogous to $e_p$ witnesses this. Thus, ${\cal A}\not\models\overline{RCPEA+P_{01}}$.
	Finally, ${\cal A}\in RCPEA_{\alpha}$ because its cylindric reduct is representable by its construction and ${\cal A}\models PE$ is checked in \cite{APEA87}.  Also, ${\cal A}\models E$ because in ${\cal A}$ there are no elements that could form the diversity atoms of a Lyndon algebra, see the proof of Lemma \ref{l:r}. Hence, ${\cal A}\models\overline{RCPEA+E}$.
	\end{proof}

 We say that in ${\cal A}$ the transposition operation $p_{kl}$ is {\it representable} iff ${\cal A}$ is isomorphic to an algebra in which all the cylindric operations $+,-,c_i,d_{ij}$ for $i,j<\alpha$ as well as $p_{kl}$ are set operations on a set $U$.  We  have seen that in the algebras ${\cal A}_p$, defined in Section \ref{s:3}, $p_{01}$ is not representable. 
 In fact, none of the nontrivial transposition operations is representable in ${\cal A}_p$, where we call $p_{kl}$ {\it nontrivial} if $k\ne l$. This is so necessarily by the following theorem.%
 \footnote{The present proof of Theorem \ref{all-thm} was suggested by one of the referees, the original proof was metalogical and longer.} 

\begin{theorem}\label{all-thm} 
Either all nontrivial transposition operations are representable in an ${\cal A}\models$ {\rm (P1)--(P6)} or none of them is.
\end{theorem}

\begin{proof}
Let $U$ be a set. We say that $p_{kl}$ is represented in ${\cal A}$ with base set $U$ if in ${\cal A}$ all the cylindric operations as well as $p_{kl}$ are set operations on $U$.  In particular, $a\subseteq{}^{\alpha}U$ for all $a\in A$. We will often omit reference to the base set $U$. 

Let $k,l,m,n<\alpha$ and assume that $p_{kl}$ is represented in ${\cal A}$. We prove that $p_{mn}$ is represented in an isomorphic copy of ${\cal A}$, on the same base set. Let ${\cal C}$ denote the full set algebra with base $U$, i.e., ${\cal C}=\langle {\cal P}({}^\alpha U), +,-,C_i,D_{ij}, P_{ij}\rangle_{i,j<\alpha}$. Note that {\rm (P1)--(P6)} are true in ${\cal C}$. 

For $i,j<\alpha$ define $h_{ij}:A\to  {\cal P}({}^\alpha U)$ by $h_{ij}(a)=P_{ij}p_{ij}a$ and let $\tau=[i,j]$.  We will show that $p_{\tau k\tau l}$ is represented in $h_{ij}({\cal A})$, the image of ${\cal A}$ under $h_{ij}$.  
Indeed, $h_{ij}$ is injective because (P6) holds both in ${\cal A}$ and in ${\cal C}$. It is a homomorphism with respect to the cylindric operations by (P1)--(P4). For example, let $a\in A$ and $q<\alpha$, then 
 
      \bigskip
      \begin{tabular}{lcl}
      $h_{ij}(c_qa)$ & = &by the  definition of $h_{ij}$\\
      $P_{ij}p_{ij}c_qa$& = & by (P3) for ${\cal A}$\\
      $P_{ij}c_{\tau q}p_{ij}a$& = &by $c_{\tau q}=C_{\tau q}$\\
      $P_{ij}C_{\tau q}p_{ij}a$& =& by (P3) for ${\cal C}$\\
      $C_{\tau\tau q}P_{ij}p_{ij}a$& =& by $\tau\tau q=q$ and the definition of $h_{ij}$\\
      $C_{q}h_{ij}(a)$.
      \end{tabular}
      \bigskip
      
      \noindent
      Similarly,  $h_{ij}$ takes $p_{\tau k\tau l}$ to $P_{\tau k\tau l}$ by (P5):
      \bigskip
      
      \begin{tabular}{lcl}
      $h_{ij}(p_{\tau k\tau l}a)$& =& \mbox{by the definition of $h_{ij}$}\\
      $P_{ij}p_{ij}p_{\tau k\tau l}a$& = & \mbox{by (P5) for ${\cal A}$}\\
      $P_{ij}p_{\tau\tau k\tau\tau l}p_{ij}a$& = & \mbox{by $\tau\tau=Id$ and $p_{kl}=P_{kl}$}\\
      $P_{ij}P_{kl}p_{ij}a $&= & \mbox{by (P5) for ${\cal C}$}\\
      $P_{\tau k\tau l}P_{ij}p_{ij}a$& = & \mbox{by the definition of $h_{ij}$}\\
      $P_{\tau k\tau l}h_{ij}(a)$.
      \end{tabular}
      \bigskip
      
 \noindent
 Thus, $p_{\tau k\tau l}$ is represented in $h_{ij}({\cal A})$. Now, let $\sigma\in P(\alpha)$ be a permutation that takes $k,l$ to $m,n$, i.e., $\sigma(k)=m$ and $\sigma(l)=n$.  Let $\sigma$ be the composition of $[i_0,j_0], \dots, [i_q,j_q]$ and let $h$ be the composition of $h_{i_0,j_0}$, \dots, $h_{i_q,j_q}$. Then $p_{\sigma k\sigma l}$, which is $p_{mn}$,  is represented in $h({\cal A})$ over $U$.
\end{proof}

Theorem \ref{all-thm} suggests an interesting new problem about the interplay between the various transposition operations.
The problem asks whether there is a nontrivial connection between the various transposition operations which is not implied by (P1)--(P8). 

\begin{problem} \label{main-prb} Let $3\le \alpha<\omega$.
If  one, or equivalently each, of the nontrivial transposition operations in ${\cal A}\in RCPEA_{\alpha}$ is representable, is then ${\cal A}$ representable?
\end{problem}

\section{On the construction}\label{const-s}
Let $3\le\alpha<\omega$. Our construction for solving Problem C evolved through several stages in \cite{AN84, APEA87, A87b} and \cite{ATu88}. 
 The unpublished notes \cite{AN84, APEA87, A87b} each contain finite nonrepresentable $RCPEA_{\alpha}$'s. The one in \cite{AN84} works only for $\alpha\ge 4$, the one in \cite{APEA87} works also for $\alpha=3$. These show that the finitely many polyadic axioms do not axiomatize $RPEA_{\alpha}$ over $RCA_{\alpha}$.  Infinitely many algebras are presented in \cite{A87b} for $\alpha\ge 4$, a nontrivial ultraproduct of them is representable. This proves nonfinite representability of $RCPEA_{\alpha}$ over $RCA_{\alpha}$ for $\alpha\ge 4$. All the algebras in \cite{AN84, APEA87, A87b} are one-generated. The construction in the present paper supersedes the one in \cite{A87b} in that it works also for $\alpha=3$ and that it allows $n$-generated subalgebras to be  representable, for all $n\in\omega$. This makes it possible to formulate nonfinite axiomatizability in a stronger form. The announcement in \cite{ATu88} is based on the construction in the present paper.

The novelty in the above constructions is to code a simple combinatorial fact of symmetric relations as an abstract property of $P_{01}$ such that the ``code" does not hold in ${RCPEA}_\alpha$. This combinatorial fact is 
\begin{description}
\item[(**)] A symmetric bijection on a set of odd cardinality has a fixed point.
\end{description}
The set in this paper is $U_0$, and a 
fixed-point-free relation $f$ is defined indirectly via a splitting $Q_0,\dots,Q_k$ of an atom $R_0\times T$, where $R_0\subseteq U_0\times U_0$. In this paper we use Lyndon relation algebras arising from finite affine planes to fix, up to isomorphism, the cardinality of $U_0$ to be $p^2$ for prime powers $p\ge 3$, hence odd. 
 When we split $R_0\times T$ into $p-1$ parts, $f$ becomes a bijection, therefore it cannot be symmetric. However, $Q_0$ is  a fixed point of the modified transposition operation $P_{01}^*$, %
 thus $P_{01}^*$ cannot be the ``real" transposition operation. 
 When we split $R_0\times T$ into fewer parts, $f$ does not become a bijection, and so there is no clash with the odd cardinality of its domain together with representability  of $P^*_{01}$. This gives a nonrepresentable algebra with representable small subalgebras.%
\footnote{For details, see Sections \ref{s:4}, \ref{s:5}. Let $f$ be defined as $f=\{ (u,v) \mid s(0\slash u, 1\slash v)\in Q_0 \}$, where $s\in R_0\times T$ is arbitrarily chosen. In Section \ref{s:4}, where we show applicability of Lemma \ref{l:y}, we prove that $f$ is a bijection of $U_0$ if $k=p-1$, and the argument works in any set algebra which is isomorphic to ${\cal A}$.}

Using nonrepresentable algebras with representable $n$-generated subalgebras, for each finite $n$, as far as we know, occurs first in Bjarni J\'onsson's paper \cite[Thm.3.5.6]{J91} where he proves that an equational axiom set for the class of representable relation algebras has to contain infinitely many variables.%
\footnote{He took Lyndon relation algebras for this purpose. We also apply Lyndon relation algebras, but for a different purpose: in our construction they are used for ensuring that $U_0$ has odd cardinality.}
This technique is applied in \cite[Thm.1]{AAPAL97} to prove the analogous result for representable cylindric algebras, and in the present paper the technique is used in the context of polyadic algebras. 
Splitting atoms into several ``copies" is a traditional rather fruitful construction in algebraic logic. See, for example, \cite{HeMoT:II, hh, Saysplit, Saybrief}.

The constructions in \cite{AN84, APEA87, A87b} differ from the one in the present paper in the tools for making sure that $|U_0|$ is odd up to isomorphism.
In \cite{AN84}, $|U_0|=3$ and we need $\alpha\ge 4$  to have this up to isomorphism. 
In \cite{APEA87}, $|U_0|=5$ and we use the so-called pentagonal relation algebra for ensuring the same property. 
In more detail, for finite $n$, let ${\cal R}_n$ be the relation algebra generated, via relation composition, by the symmetric successor function $S_n=\{(i,i+k) \mid i<n, k\in\{ 1,-1\}\}$ where $+$ and $-$ are understood modulo $n$. The pentagonal relation algebra is ${\cal R}_5$ and it happens that all its representations have base set of cardinality 5.%
\footnote{This is well known, see, e.g., \cite[56.14]{Madd}, \cite[p.139]{hh}.}
In \cite{A87b}, $U_0$ can have any odd cardinality $n$; we use ${\cal R}_n$ to ensure this up to isomorphism as follows. The atoms of ${\cal R}_n$ are just $S_n^q=\{(i,i+k) \mid i<n, k\in\{ q,-q\}\}$ for $q<m$ where $m=\frac{n-1}{2}+1$ and we
formulate that in $S_n^q$, $0<q<m$ each $u$ is connected to exactly two other elements. We need $\alpha\ge 4$ to be able to do this by using the diagonal constants and cylindrifications. 
We note, by hindsight, that the construction in \cite{AN84} is the special case of $n=3$, i.e., it uses ${\cal R}_3$.

In all these preliminary constructions, we split the atom $S_0\times T$ into two parts, as opposed to the construction in this paper, where we split $R_0\times T$ into $p-1$ parts. This is why these algebras are all one-generated. Since they have a representable ultraproduct, they are suitable for proving that even the one-variable equations true in ${RPEA}_\alpha$ cannot be axiomatized over $RCA_\alpha$ with finitely many equations. For analogous proofs we refer to \cite[Thm.s 3,4]{Madd89}, where it is proved that the set of one-variable equations true in $RCA_\alpha$ is not finitely axiomatizable, and the same is true for relation algebras. 

J\'onsson's technique of using nonrepresentable algebras with large representable subalgebras also proves non-axiomatizability with universal prenex formulas containing only a bounded number of variables (see, e.g., \cite[p.520]{hh}), but it leaves open the possibility of axiomatizability of $RPEA_{\alpha}$ over $RCA_{\alpha}$ with a set of first-order sentences that contains only finitely many variables.

\begin{problem}\label{folax-p}
	Is there a set $\Sigma$ of first-order formulas in the language of $RPEA_\alpha$ that contains only finitely many variables and $RPEA_\alpha$ is the class of all models of $\Sigma$ that have representable cylindric reduct?
\end{problem}
\noindent Since both $RPEA_\alpha$ and $RCA_\alpha$ are axiomatizable by equations, the above problem is equivalent to finding a set $\Sigma$ of first-order formulas containing only finitely many variables and true of $RPEA_{\alpha}$ such that all equations valid in $RPEA_\alpha$ can be derived from $\Sigma$ together with the equational theory of $RCA_\alpha$.
\noindent The analogous problems for axiomatizing the classes of representable relation, cylindric, and polyadic equality algebras are open as of today, see Problem 17.14 and Problem 17.13 in \cite{hh}. A non-axiomatizability result beyond universal prenex formulas is contained in \cite{EH}.

\section{On generalizing to infinite dimensions}\label{inf-s}
In this section, we briefly describe what the axiomatizability situation is when it is not the case that $3\le\alpha<\omega$. 
For $\alpha<3$, axiomatizability problems in algebraic logic usually yield positive answers, and this is the case also for $RPEA_{\alpha}$: it is axiomatizable by a finite set of equations, see \cite[p.242]{HeMoT:II}. For infinite $\alpha$, the situation is more intricate. 
The definitions of $\Fm^+\slashm\fequiv$, $RCA_{\alpha}$, $RPEA_{\alpha}$ etc.\ are straightforward to generalize to infinite $\alpha$. However, the algebras in $RPEA_{\alpha}$ thus obtained are called {\it quasi-polyadic algebras} and their class is denoted by $RQEA_{\alpha}$. The reason is that polyadic algebras were defined to have a more complex index set for the cylindric and substitution operations by Paul Halmos \cite{Ha}. For finite $\alpha$, this does not make much difference, because the originally defined class is term-definitionally equivalent to the one used in this paper, but for infinite $\alpha$ there is a big gap between polyadic algebras and quasi-polyadic algebras. For details, see \cite{SaTho}.  The proofs given here for Proposition \ref{pse-prop} and Theorems \ref{poz-thm}, \ref{pea-p} in the next section generalize to infinite $\alpha$ almost verbatim.

Assume $\alpha\ge\omega$ from now on, but only in this section. 
The problems of axiomatizability by finitely many equations have trivially negative answers in this case, because our algebras have an infinity of operations. 
However, (P1)--(P8) are meaningful in this case, too: they are eight schemes that define an infinity of equations, and the question whether they axiomatize the transposition operations is meaningful.
It is proved in \cite{ANSay} that (P1)--(P8)  do not axiomatize the transposition operations over the cylindric ones.%
\footnote{The construction used in \cite{ANSay} is the one in \cite{AN84} generalized to infinite dimensions.} 
Monk \cite{Mo} proved that $RCA_{\alpha}$ cannot be axiomatized by finitely many equation schemes like (P1)--(P8). For this, he had to define what an ``equation scheme like (P1)--(P8)" means, they are called ``Monk schemata" in \cite{SaTho}. In respect to Monk schemata, the situation for quasi-polyadic algebras seems to be analogous to the case of finite $\alpha$: Ildik\'o Sain and Richard Joseph Thompson give several proof ideas for non-axiomatizability of $RQEA_{\alpha}$ by finitely many Monk schemata, see \cite[Thm.2.1(ii)]{SaTho}, and Tarek Sayed  Ahmed gives a proof outline for non-axiomatizability of $RQEA_{\alpha}$ over $RCA_{\alpha}$ by finitely many Monk schemata, see \cite[Thm.3.6]{Sayyetsome}.%
\footnote{We believe that the proof idea using ultraproducts in \cite{SaTho} can be elaborated to work fine, and the outline in \cite{Sayyetsome} also can be made into a proof if one uses the constructions in \cite{A87b}.}
Non-axiomatizability with finitely many variables is a stronger property than non-axiomatizability by finitely many Monk schemata. It is proved in \cite[Thm.1]{AAPAL97} that $RCA_{\alpha}$ cannot be axiomatized by a set of equations containing only finitely many variables. The situation for quasi-polyadic algebras is not known:

\begin{problem}\label{qpea-prob} Assume $\alpha\ge\omega$. 
	\begin{description}
		\item[(i)] Can $RQEA_{\alpha}$ be axiomatized by a set of equations containing only finitely many variables?
		\item[(ii)] Can $RQEA_{\alpha}$ be axiomatized over $RCA_{\alpha}$ by a set of equations containing only finitely many variables?
	\end{description}	
\end{problem}
\noindent We note that Sain and Thompson give in \cite{SaTho} a proof idea for a negative answer to Problem \ref{qpea-prob}(i), but that idea does not work because the algebras they use have two-generated nonrepresentable subalgebras. This is so because one can use the transposition operations 
to code any finitely many ``split" elements in that construction into one element such that the finitely many can be regained from that one element and a split element.%
\footnote{Namely,  $Q_i=c_0\dots c_nQ_0\cdot p_{0i}\sum\{ p_{0j}Q_j : j\le n\}$, for $i\le n$.}
This was discovered by Sayed Ahmed who tried to overcome this difficulty  by using another method.
However, his proof does not work, either.%
\footnote{\cite[Thm.2.3]{Sayyetsome} states a positive answer to Problem \ref{qpea-prob}(i), but the proof contains an error in the last but one paragraph on p.341: there is no isomorphism the existence of which is stated there.}

\section{Complexity of proof systems for finite variable logic; a transparent polyadic axiom set}
\label{s:7}
\label{log-s}

In this closing section, the first two theorems are about (P1)--(P8), in particular they are utilized to present a new transparent axiom set for polyadic equality algebras of finite dimension. Further, 
a theorem about complexity of complete derivation systems for the finite-variable fragment of first-order logic is proved. 
We assume  $3\le\alpha<\omega$ in this section, but  almost everything in it generalizes to infinite $\alpha$.

Recall from Section \ref{s:1} that ${\Fm}^+$ denotes the formula-algebra of $\alpha$-variable first-order logic together with the transposition operations $p_{ij}$ as concrete functions on it. For the next theorem to be true, no assumption on relation symbols is needed.

\begin{theorem}\label{poz-thm}
	An equation is true in $\Fm^+$ if and only if it is a logical consequence of {\rm (P1)--(P7)}.
\end{theorem}

\begin{proof} Clearly, (P1)--(P7) are true in $\Fm^+$, by the definitions of $p_{ij}$ in this algebra. Assume $\Fm^+\models \sigma=\tau$, we show that $\sigma=\tau$ is derivable from (P1)--(P7).
	
By using (P1)--(P4), we can ``push in" all the $p_{ij}$ in a term to stand in a queue in front of some algebraic variable symbol. 
Let us call a term {\it terminal} if it is of form $p_{i_1j_1}\dots p_{i_nj_n}x$,  and let us call a term {\it normal} if it is built up from terminal terms by the use of $+,-,c_i,d_{ij}$, $i,j<\alpha$.
Thus, by the use of (P1)--(P4), each term provably is equal to a normal one. 

Hence, we may assume that $\sigma$ and $\tau$ are normal. Assume that $\sigma$ is of form $-\sigma^{\prime}$. Then, by  $\Fm^+\models\sigma=\tau$, also $\tau$ has to be of form $-\tau^{\prime}$ and $\Fm^+\models\sigma^{\prime}=\tau^{\prime}$, because $\Fm$ is a term algebra. The analogous facts are true for $+,c_i,d_{ij}$, so we may assume that both $\sigma$ and $\tau$ are terminal terms.

Assume that $\sigma$ is $p_{i_1j_1}...p_{i_nj_n}x$ and $\tau$ is $p_{k_1l_1}...p_{k_ml_m}y$ for some $i_1,\dots,l_m<\alpha$ and algebraic variables $x,y$. By $\Fm^+\models\sigma=\tau$, the variables $x$ and $y$ have to be the same, this can be seen by evaluating $x$ to any formula $\varphi\in Fm$ and $y$ to $\lnot \varphi\in Fm$; we can do this if $x$ and $y$ are distinct variables.  Let $\rho, \eta$ be $[i_1,j_1]\circ\dots\circ [i_n,j_n]$ and $[k_1,l_1]\circ\dots\circ [k_m,l_m]$, respectively. Then $\rho$ and $\eta$ are permutations of $\alpha$, and $\sigma$ and $\tau$ evaluate to $v_{\rho0}=v_{\rho1}\lor\dots\lor v_{\rho0}=v_{\rho(\alpha-1)}$ and to $v_{\eta0}=v_{\eta1}\lor\dots\lor v_{\eta0}=v_{\eta(\alpha-1)}$, respectively, when $x$ is evaluated to $v_{0}=v_{1}\lor\dots\lor v_{0}=v_{\alpha-1}$. Thus, by $\Fm^+\models\sigma=\tau$ we have that $\rho0=\eta0$, \dots, $\rho(\alpha-1)=\eta(\alpha-1)$, i.e., $\rho=\eta$ in $P(\alpha)$.

Thus,  $[i_1,j_1]\circ\dots\circ [i_n,j_n]=[k_1,l_1]\circ\dots\circ [k_m,l_m]$ is true in the semigroup of transpositions, so it can be derived from J\'onsson's three defining relations (J1)--(J3) in \cite{J62}:
\begin{description}
	\item[(J1)] $[i,j]=[j,i]$, 
	\item[(J2)] $[i,j]\circ[i,j]=Id$, 
	\item[(J3)] $[i,j]\circ[i,k]=[j,k]\circ[i,j]$, 
\end{description}
where $i,j,k<\alpha$ are distinct. Now, 
\begin{description}
	\item[(P9)] $p_{ij}x=p_{ji}x$
\end{description}
can be derived from  (P5)--(P7) as follows. First we get $p_{ij}p_{ji}x=p_{ij}p_{ij}x=x$ by substituting $j,i$ in place of $k,l$ in (P5) (notice that in (P1)--(P7) $i,j,k,l$ do not have to be distinct) and then applying (P6).  
Then we get $p_{ji}x=p_{ij}x$ by applying $p_{ij}$ to both sides and using (P6). 
Notice that (J3) is the special case of (P5) where we substitute $i,k$ in place of $k,l$ respectively. 
Let us take a (J1)--(J3) derivation of $[i_1,j_1]\circ\dots\circ [i_n,j_n]=[k_1,l_1]\circ\dots\circ [k_m,l_m]$ and transform it to a (P1)--(P7) derivation of $p_{i_1j_1}...p_{i_nj_n}x=p_{k_1l_1}...p_{k_ml_m}x$ by applying everywhere
 (P9), (P6), (P5),  respectively, where (J1), (J2), (J3), respectively, is used in the (J1)--(J3) derivation.  

We have seen that $\sigma=\tau$ can be derived from (P1)--(P7).	
\end{proof}

Theorem \ref{poz-thm} suggests a new transparent axiom set for polyadic equality algebras of finite dimension. Let CAx be any equational axiom set for cylindric algebras, in the language $+,-,c_i, d_{ij}$, $i,j<\alpha$. For concreteness, let us take axioms (C0)--(C11)  from \cite[p.\ 161]{HeMoT:II}. We do not recall these axioms because we concentrate now on the transposition operations $p_{ij}$. Define
\[ TPEAx = CAx \cup \{ (P1),\dots, (P8)\},\]
and let $TPEA_{\alpha}$ be the class of algebras in which $TPEAx$ is true. Let $PEA_{\alpha}$ denote the class of polyadic equality algebras of dimension $\alpha$ as defined in%
\footnote{On various different style  presentations of polyadic algebras see, e.g., \cite{Cir88, Cir19, Ha, Pi73, Pi73a}.} \cite[Def.\ 5.4.1]{HeMoT:II} and in \cite[Def.\ 5]{SaTho}.  Term-definitional equivalence between classes of algebras is a close connection, for the notion see \cite[0.4.14]{HeMoT:II} or \cite[p.547]{SaTho}.

\begin{theorem}\label{pea-p}  $TPEA_{\alpha}$ is term-definitionally equivalent to $PEA_{\alpha}$.
\end{theorem}

\begin{proof} It is proved in \cite[Theorem 1]{SaTho} that $PEA_{\alpha}$ is term-definitionally equivalent, for finite $\alpha$, to the class $FPEA_{\alpha}$ of finitary polyadic equality algebras.%
	\footnote{\cite[Theorem 1]{SaTho} is stated for a class $QPEA_{\alpha}$, but $QPEA_{\alpha}$ coincides with $PEA_{\alpha}$ for finite $\alpha$.} Therefore, it is enough to  prove that $TPEA_{\alpha}$ is term-definitionally equivalent to $FPEA_{\alpha}$. The language of $FPEA_{\alpha}$ is that of $TPEA_{\alpha}$ together with unary operations $s^i_j$ for $i,j<\alpha$, and $FPEA_\alpha$ is  defined by the following equations, where $i,j,k<\alpha$.
	\begin{description}
		\item[(F0)] $+,-$ form a Boolean algebra, $s^i_ix=p_{ii}x=x$, and $p_{ij}x=p_{ji}x$, 
		\item[(F1)] $x\le c_ix$, 
		\item[(F2)] $c_i(x+y)=c_ix+c_iy$, 
		\item[(F3)] $s^i_jc_ix = c_ix$, 
		\item[(F4)] $c_is^i_jx = s^i_jx$\quad if $i\ne j$, 
		\item[(F5)] $s^i_jc_kx = c_ks^i_jx$\quad if $k\notin\{ i,j\}$,
		\item[(F6)] $s^i_j$ and $p_{ij}$ are Boolean endomorphisms (i.e., $s^i_j(-x) = -s^i_jx$  etc.), 
		\item[(F7)] $p_{ij}p_{ij}x = x$, 
		\item[(F8)] $p_{ij}p_{ik}x = p_{jk}p_{ij}x$\quad if $i,j,k$ are all distinct, 
		\item[(F9)] $p_{ij}s^j_ix = s^i_jx$, 
		\item[(F10)] $s^i_jd_{ij} = 1$, 
		\item[(F11)] $x\cdot d_{ij}\le s^i_jx$.		
	\end{description}
Let us take the interpretation of $FPEA_{\alpha}$ into $TPEA_{\alpha}$ where all operations of $FPEA_{\alpha}$, except for the $s^i_j$, are interpreted to themselves, and we interpret $s^i_j$ as
\[ s^i_ix = x\quad\mbox{ and }\quad s^i_jx=c_i(d_{ij}\cdot x)\quad\mbox{if }i\ne j .\]
We show that (F0)--(F11) hold under this interpretation. Indeed, (F0) holds by  CAx, (P7) and (P9); (F1)--(F5), (F10)--(F11) and the $s^i_j$-part of (F6) follow from CAx; the $p_{ij}$-part of (F6) is (P1)+(P2); and (F7), (F8) are (P6), (P5), respectively. It remains to derive (F9).	Now, for $i\ne j$ we have
$p_{ij}s^j_ix = p_{ij}(c_j(d_{ji}\cdot x))= c_ip_{ij}(d_{ji}\cdot x) = c_i(d_{ij}\cdot x) = s^i_jx$, by the definition of $s^j_i$, (P3), (P8) and CAx, and the definition of $s^i_j$, respectively. 

For the interpretation of $TPEA_{\alpha}$ in $FPEA_{\alpha}$ let us just ``forget" the operations $s^i_j$. We have to show that CAx together with (P1)--(P8) hold in $FPEA_{\alpha}$.  Now, CAx holds in $PEA_{\alpha}$ by \cite[5.4.3]{HeMoT:II}, and in \cite{SaTho} the interpretation of $FPEA_{\alpha}$ in $PEA_{\alpha}$  is such that the cylindric operations $+,-,c_i,d_{ij}$ are interpreted by themselves. Thus CAx holds in $FPEA_{\alpha}$.
(P1)--(P8) hold in $FPEA_{\alpha}$ by the following. (P1), (P2) follow from (F6); and (P6), (P7) follow from (F7), (F0). (P3) and (P4) follow from (Q9) and (E3) of \cite{SaTho}, respectively. (P5) follows from \cite[Claim 1.2]{SaTho} and checking $[i,j]\circ[k,l] = [\tau(k),\tau(l)]\circ[i,j]$ where $\tau=[i,j]$. 
If $i=j$ then $(P8)$ holds trivially by (P7), which we already have checked. Assume $i\ne j$, then (P8) follows from (F9) by the following. The operation $s^i_j$ of $FPEA_{\alpha}$ is interpreted in $PEA_{\alpha}$ by the operation $s_{[i/j]}$ and \cite[5.4.3]{HeMoT:II} states that $s_{[i/j]}x=c_i(d_{ij}\cdot x)$ holds in $PEA_{\alpha}$. Hence, $s^i_jx=c_i(d_{ij}\cdot x)$ holds in $FPEA_{\alpha}$. Now $p_{ij}s^j_ix = s^i_jx$ by (F9), so $(p_{ij}s^j_ix)\cdot d_{ij} = (s^i_jx)\cdot d_{ij}$. Note that CAx implies $c_i(x\cdot d_{ij})\cdot d_{ij}=x\cdot d_{ij}$, see \cite[1.3.9]{HeMoT:II}.
Now, 
$(p_{ij}s^j_ix)\cdot d_{ij} = p_{ij}c_j(x\cdot d_{ij})\cdot d_{ij} = c_ip_{ij}(x\cdot d_{ij})\cdot d_{ij} = c_i(p_{ij}x\cdot d_{ij})\cdot d_{ij} = p_{ij}x\cdot d_{ij} = p_{ij}(x\cdot d_{ij})$.  Similarly,
$s^i_jx\cdot d_{ij} = c_i(x\cdot d_{ij})\cdot d_{ij} = x\cdot d_{ij}$, and we are done with showing that (P8) holds in $FPEA_{\alpha}$. 

For showing term-definitional equivalence of $TPEA_{\alpha}$ and $FPEA_{\alpha}$ it is enough to show that the two interpretations above are inverses of each other. This follows from $FPEA_{\alpha}\models s^i_j(x)=c_i(d_{ij}\cdot x)$ for distinct $i,j$.
\end{proof}

\bsk

We now state a theorem about complexity of complete derivation systems for finite-variable logic. In the presentation, we shall closely follow Monk's paper \cite{mo2}.

Let us recall $\alpha$-variable first-order logic $L_{\alpha}$. There are infinitely many ${\alpha}$-place relation symbols $R_i$, $i<\omega$ in the language. The set of variables is $V=\{ v_0,\dots,v_{\alpha-1}\}$. The relational atomic formulas are $R_k(v_{i_1},\dots,v_{i_{\alpha}})$ for $k<\omega$ and $i_1,\dots,i_{\alpha}<\alpha$, the rest of the atomic formulas are $v_i=v_j$ for $i,j<\alpha$. A formula is built up from atomic formulas by the use of the unary logical connectives $\lnot$ and $\exists v_i$ together with the binary logical connective $\lor$. We use derived logical connectives, such as $\land, \to, \forall v_i$, as is usual. The set of formulas is denoted by $Fm$. 
{\it Detachment} is the derivation rule according to which we can infer $\psi$ from $\{\varphi, \varphi\to\psi\}$, for any $\varphi,\psi\in Fm$. {\it Generalization} is the derivation rule according to which we can infer $\forall v_i\varphi$ from any formula $\varphi\in Fm$  and $i<\alpha$. 

We will also use a derivation rule {\it Inst} which allows substituting arbitrary formulas in place of the relational atomic formulas.%
\footnote{The more derivation rules we use, the stronger Theorem \ref{log-t} will be.} The only requirement is that if we replace, say, $R_i(v_0,\dots,v_{\alpha-1})$  with $\varphi$, then we have to replace $R_i(v_{\tau 0},\dots,v_{\tau(\alpha-1)})$ by a version $S(\tau)\varphi$ of $\varphi$ in which we systematically replace the variables $v_j$ by $v_{\tau j}$. Since here $\tau$ may not be a bijection and we want to be compatible with semantics, the substitution involves renaming of bound variables. We recall the definition of $S(\tau)\varphi$ from \cite{mo2}.
\begin{description}
	\item{} $S(\tau)R(v_{i_1},\dots,v_{i_{\alpha}})$\ \  is\ \ $R(v_{\tau(i_1)},\dots,v_{\tau(i_{\alpha})})$, \quad $S(\tau)v_i=v_j$\ \  is\ \  $v_{\tau(i)}=v_{\tau(j)}$,
	\item{} $S(\tau)\lnot\varphi$\ \ is\ \ $\lnot S(\tau)\varphi$,\quad  $S(\tau)(\varphi\lor\psi)$\ \ is\ \ $S(\tau)\varphi\lor S(\tau)\psi$,\quad and
	\item{} $S(\tau)\exists v_i\varphi$\ \ is\ \ $\exists v_j S(\sigma)\varphi$\qquad where $j$ is the least element of $\alpha-\{\tau(k)\mid i\ne k<\alpha\}$ and $\sigma(i)=j$, $\sigma(k)=\tau(k)$ for all $k\ne i$.
\end{description}
Now, if $\delta$ is any assignment of formulas to relation symbols, an {\it instance} of $\varphi$ is the formula we get from $\varphi$ by substituting $S(\tau)\delta(R_i)$ in place of $R_i(v_{\tau(0)},\dots,v_{\tau(\alpha-1)})$, simultaneously for all relation symbols $R_i$. If $\Gamma$ is a set of formulas, then $Inst(\Gamma)$ denotes the set of all instances of members of $\Gamma$.

We do not recall here the semantical notions of models and truth. A formula is called valid, or a {\it tautology}, iff it is true in each model under each evaluation of the variables. A {\it complete axiom system} for $L_{\alpha}$ is a set $\Gamma$ of tautologies such that all other tautologies can be derived by a series of use of Detachment and Generalization from $Inst(\Gamma)$. The theorem in \cite{mo2} states that a complete axiom system has to be infinite. The following theorem states a stronger property of complete axiom systems.

\begin{theorem}\label{log-t} Any complete axiom system for $L_{\alpha}$ must contain for any $n<\omega$ a formula $\varphi$ with the following three properties:
	\begin{description}
		\item[(i)] at least $n$ distinct relation symbols occur in $\varphi$,
		\item[(ii)] both $R_i(v_{\tau 0},\dots,v_{\tau(\alpha-1)})$ and $R_i(v_{\sigma 0},\dots,v_{\sigma(\alpha-1)})$ occur in $\varphi$, for some $i<\omega$ and distinct permutations $\tau, \sigma$ of $\alpha$,
		\item[(iii)] existential quantifier $\exists$, equality symbol $=$ and disjunction $\lor$ all occur in $\varphi$.
		\end{description}	
\end{theorem}

\begin{proof}
	The argument follows \cite{mo2}, except that, in place of Johnson's theorem in \cite{Jo} we use our stronger Theorem \ref{t:1}. 
	Assume that $\Gamma$ is a complete axiom system for $L_{\alpha}$. 
	
	We translate $\Gamma$ into  a set  $Eq(\Gamma)$ of equations and we prove that $Eq(\Gamma)$ is an axiomatization for $Pse_{\alpha}$. Let $X=\{ x_i \mid i<\omega\}$ be a set of algebraic variables. We define, for any $\varphi\in Fm$ a term $T(\varphi)$ in the algebraic language of $Pse_{\alpha}$ as follows.
 Recall that a transposition [i,j] is the bijection of $\alpha$ which interchanges $i$ and $j$ and leaves all other elements of $\alpha$ fixed, and a replacement [i/j] is the function that takes $i$ to $j$ and leaves all other elements of $\alpha$ fixed. Let $\tau:\alpha\to\alpha$. It is known that if $\tau$ is a bijection, then it is a composition of transpositions, and if $\tau$ is not a bijection then it is a composition of replacements; see, e.g., \cite{J62} and \cite{Tho}. For any $\tau$ fix such a sequence of transpositions or replacements. Let $i<\omega$. Assume that $\tau=[i_1,j_1]\circ\dots\circ[i_k,j_k]$, then
\[ T(R_i(v_{\tau 0},\dots,v_{\tau(\alpha-1)}))\quad\mbox{ is }\quad p_{i_1j_1}...p_{i_kj_k}x_i ,\]
assume that $\tau=[i_1/j_1]\circ\dots\circ[i_k/j_k]$, then
\[ T(R_i(v_{\tau 0},\dots,v_{\tau(\alpha-1)}))\quad\mbox{ is }\quad s_{i_1j_1}...s_{i_kj_k}x_i ,\]
where $s_{ij}x = c_i(d_{ij}\cdot x)$ if $i,j$ are distinct. Let $i,j<\alpha$. Then $T(v_i=v_j)$ is $d_{ij}$, $T(\lnot\varphi)$ is $-T(\varphi)$, $T(\varphi\lor\psi)$ is $T(\varphi)+T(\psi)$, and $T(\exists v_i\varphi)$ is $c_iT(\varphi)$. By this, the function $T$ has been defined.  One can prove, by induction, just as in \cite{mo2}, that $\models\varphi$ iff $Pse_{\alpha}\models T(\varphi)=1$, where $1$ is the Boolean constant. We define%
\footnote{We note that in \cite{mo2} there is a typo in the definition of the analogous set of equations. Namely, the polyadic equations have to be added because they are used in the proof of Lemma 14 there.}
\[ Eq(\Gamma) = \{ T(\varphi)=1 \mid \varphi\in\Gamma\} \cup TPEAx .\]
Now, $Pse_{\alpha}\models Eq(\Gamma)$.

Let ${\cal A}\models Eq(\Gamma)$, we will show that ${\cal A}$ is representable. Define $\Sigma=\{\varphi \mid {\cal A}\models T(\varphi)=1\}$. Then $\Gamma\subseteq\Sigma$, by our assumption. One can see that $\varphi\in\Sigma$ implies $Inst(\varphi)\subseteq\Sigma$, by the definition of an equation being true in an algebra. Finally, $\Sigma$ is closed under Detachment and Generalization, by $TPEAx\subseteq Eq(\Gamma)$. Hence, $\Sigma$ contains all tautologies, by our assumption that $\Gamma$ is complete. We now will use another translation function that is more or less the inverse of $T$. First we translate normal terms in the language of $Pse_{\alpha}$ to formulas. 
Assume that $\tau=[i_1,j_1]\circ\dots\circ[i_k,j_k]$, then
\[ F(p_{i_1j_1}...p_{i_kj_k}x_i) \quad\mbox{ is }\quad R_i(v_{\tau 0},\dots,v_{\tau(\alpha-1)}),\]
and $F(d_{ij})$ is $v_i=v_j$, $F(-\sigma)$ is $\lnot F(\sigma)$, $F(\sigma+\delta)$ is $F(\sigma)\lor F(\delta)$, and $F(c_i\sigma)$ is $\exists v_i F(\sigma)$. One can prove that $\models F(\sigma)$ iff $Pse_{\alpha}\models\sigma=1$, and $TPEAx\models \sigma=TF(\sigma)$. 
Now, to show that ${\cal A}$ is representable, it is enough to show that ${\cal A}\models\sigma=1$ whenever $Pse_{\alpha}\models\sigma=1$ and $\sigma$ is normal.
Indeed, $Pse_{\alpha}\models\sigma=1$ implies that $\models F(\sigma)$, which implies that ${\cal A}\models TF(\sigma)=1$, i.e., ${\cal A}\models\sigma=1$ by ${\cal A}\models TPEAx$. 
We have seen that $Eq(\Gamma)$ is an equational axiom set for $Pse_{\alpha}$. We can use now Theorem \ref{t:1}.

Let us say that $R_i$ occurs in $\varphi$ twice, if both $R_i(v_{\tau0},\dots,v_{\tau(\alpha-1)})$ and $R_i(v_{\sigma0},\dots,v_{\sigma(\alpha-1)})$ occur in $\varphi$ for distinct bijections $\tau,\sigma$ of $\alpha$. For any formula $\varphi$, let $\overline{\varphi}$ denote the formula we obtain from $\varphi$ by replacing $R_i(v_{\tau0},\dots,v_{\tau(\alpha-1)})$ in it with $R_i(v_0,\dots,v_{\alpha-1})$ whenever $R_i$ does not occur twice in $\varphi$, simultaneously, and let $\overline{\Gamma}=\{\overline{\varphi} \mid \varphi\in\Gamma\}$. Note that transposition operations do not occur in $T(\overline{\varphi})$ if there is no $R_i$ that occurs twice in $\varphi$. Now, $Inst(\varphi)=Inst(\overline{\varphi})$ by the definition of $Inst$, hence $\overline{\Gamma}$ is also a complete axiom system for $L_{\alpha}$. Let $n\ge 3$ and let us call a formula {\it complex} if (i), (iii) of Theorem \ref{log-t} hold for it. Assume that there is no complex formula in $\Gamma$ in which some $R_i$ occurs twice. Then there is no equation $e_n$ in $Eq(\overline{\Gamma})$ with the properties required by  Theorem \ref{t:1}, which is a contradiction since we have seen that $Eq(\overline{\Gamma})$ is an equational axiom set for $Pse_{\alpha}$. So, there is a complex formula $\varphi\in\Gamma$ in which some $R_i$ occurs twice, and we are done.
\end{proof}

For the next theorem to be true, it is necessary to assume that there are infinitely many $\alpha$-place relation symbols in the logical language. Proposition~\ref{pse-prop} below is basically known in algebraic logic, see \cite{mo2}, \cite[Sec.\ 4.3]{HeMoT:II}  and \cite[Ex.\ 5.3]{AGyNS}. We include a proof because this proposition establishes the connection between our logical and algebraic results, in particular, it was used in proving Theorem~\ref{th:0} from Theorem \ref{t:1} (see Section \ref{s:1}). Recall that $\equiv$ was defined just before (P8) in Section \ref{s:1}: $\varphi\equiv\psi$ means that $\varphi\leftrightarrow\psi$ is valid.

\begin{prop}\label{pse-prop}
	An equation is true in \mbox{$\Fm^+\slashm\fequiv$} if and only if it is true in $Pse_{\alpha}$, and an equation is true in \mbox{$\Fm\slash\fequiv$} iff it is true in $Cs_{\alpha}$.
\end{prop}

\begin{proof}  The $\Fm\slash\fequiv$ part of the theorem immediately follows from 3.3.12 and 5.4.1 of \cite{AGyNS}, or from \cite[Sec.\ 4.3]{HeMoT:II}. The other part is not so immediate, because $\Fm^+\slashm\fequiv$ is not a Lindenbaum-Tarski algebra, since the $p_{ij}$ operations do not come from connectives in a logic. We shall use the terminology of \cite[Section~3]{AGyNS}. Our logic is Example 5.3 there, and it is called ${\alpha}$-variable logic with substituted atomic formulas. 
	Let ${\cal R}$ denote the set of relation symbols, then a model is $\Mm=\langle U, R^{\Mm}\rangle_{R\in{\cal R}}$. For a model like this, let its {\it meaning function} be defined as
	$$ mng_{\Mm}(\varphi) = \{ s\in U^{\alpha} \mid \Mm\models\varphi[s]\} ,$$
	see, e.g., \cite{mo2}. It is routine to check that $mng_{\Mm}:\Fm^+\to{\cal A}$ is a homomorphism, where ${\cal A}$ is a polyadic equality set algebra with base set $U$. By definition, $\equiv$ is the intersection of the kernels of all the meaning functions. So, if an equation is true in $Pse_{\alpha}$ then it is true in $\Fm^+\slashm\fequiv$, too.	
	
	Assume  now that an equation $e$ fails in $Pse_{\alpha}$, say, ${\cal A}\not\models e[k]$ with the evaluation $k$ of the algebraic variables $\{ x_1,\dots,x_n\}$ occurring in $e$ in ${\cal A}\in Pse_{\alpha}$ with base set $U$. Take any model $\Mm=\langle U,R^{\Mm}\rangle_{R\in{\cal R}}$ such that $R_i^{\Mm}=k(x_i)$ for $1\le i\le n$, for some system $R_i, 1\le i\le n$ of $\alpha$-place relation symbols. There is such a model by our assumption on having infinitely many $\alpha$-place relation symbols. Take an evaluation of the algebraic variables in $\Fm^+\slashm\fequiv$ such that $h(x_i)=R_i(v_0,\dots,v_{\alpha-1})\slash\fequiv$ for $1\le i\le n$.  Now, $mng_{\Mm}$ induces a homomorphism  $m:\Fm^+\slashm\fequiv\,\,\to{\cal A}$ such that $m(R_i(v_0,\dots,v_{\alpha-1})\slash\fequiv)\,\, = k(x_i)$, and $k=m\circ h$. Thus $\Fm^+\slashm\fequiv\,\,\not\models e[h]$ by ${\cal A}\not\models e[k]$.	
\end{proof}

\section*{Acknowledgement} We are grateful for the many useful suggestions, remarks and corrections that the two referees provided.

\end{document}